\documentclass[12pt]{amsart} 
\usepackage{amsmath, amsthm, amssymb, mathtools, wasysym, verbatim, bbm, color, graphics, geometry,graphicx,algorithm2e,accents,enumitem,float,subfigure}
\usepackage{todonotes}

\usepackage[compress,noadjust]{cite}

\newcommand{\bR}{\mathbb{R}}

\newcommand\cB{\mathcal{B}}
\newcommand\cC{\mathcal{C}}

\newcommand\cL{\mathcal{L}}

\newcommand\cS{\mathcal{S}}

\newcommand\tr{\operatorname{tr}}
\newcommand\dist{\operatorname{dist}}
\newcommand{\p}{\partial}
\newcommand{\epsi}{\varepsilon}

\newcommand\restr[2]{{
  \left.\kern-\nulldelimiterspace 
  #1 
  \littletaller 
  \right|_{#2} 
  }}

\newcommand{\av}{A[v]}

\newcommand{\sff}{\Pi}

\newcommand{\bl}{\lambda^*}
\newcommand{\hrn}{\bR^n_+}

\theoremstyle{plain}
\newtheorem{theorem}{Theorem}[section]
\newtheorem{lemma}{Lemma}[section]
\newtheorem{corollary}{Corollary}[section]
\newtheorem{proposition}{Proposition}[section]

\theoremstyle{definition}

\newtheorem{question}{Question}[section]

\theoremstyle{remark}
\newtheorem{remark}{Remark}[section]

\makeatletter
\def\thmhead@plain#1#2#3{%
  \thmname{#1}\thmnumber{\@ifnotempty{#1}{ }\@upn{#2}}%
  \thmnote{ {\the\thm@notefont#3}}}
\let\thmhead\thmhead@plain
\makeatother

\begin{document}

\title[Liouville Theorem On the Half Euclidean Space]{Liouville Theorem with Boundary Conditions from Chern--Gauss--Bonnet Formula}

\author[B.Z. Chu]{BaoZhi Chu}
\address[B.Z. Chu]{Department of Mathematics, Rutgers University, 110 Frelinghuysen Road, Piscataway, NJ 08854-8019, USA}
\email{bc698@math.rutgers.edu}

\author[Y.Y. Li]{YanYan Li}
\address[Y.Y. Li]{Department of Mathematics, Rutgers University, 110 Frelinghuysen Road, Piscataway, NJ 08854-8019, USA}
\email{yyli@math.rutgers.edu}

\author[Z. Li]{Zongyuan Li}
\address[Z. Li]{Department of Mathematics, City University of Hong Kong, Hong Kong SAR}
\email{zongyuan.li@cityu.edu.hk}

\begin{abstract}
 The $\sigma_k(A_g)$ curvature and the boundary $\cB^k_g$ curvature arise naturally from the Chern--Gauss--Bonnet formula for manifolds with boundary.
In this paper, 
    we prove a Liouville theorem for the equation $\sigma_k(A_g)=1$ in $\overline\hrn$ with the boundary condition $\cB^k_g=c$ on $\p\hrn$, where $g=e^{2v}|dx|^2$ and $c$ is some nonnegative constant. This extends an earlier result of Wei, which assumes the existence of $\lim_{|x|\to\infty}(v(x)+2\log|x|)$. In addition, we establish a local gradient estimate for solutions of such equations, assuming an upper bound on the solution $v$.
\end{abstract}

\maketitle
	
\vspace{.25in}
	
\section{Introduction}
On a Riemannian manifold $(M^n,g)$ with boundary $\p M$ of dimension $n\geq 3$,
consider the Schouten tensor
\begin{equation*}
    A_g=\frac{1}{n-2}\big( Ric_g-\frac{R_g}{2(n-1)}g\big),
\end{equation*}
where $Ric_g$ and $R_g$ denote, respectively, the Ricci tensor and the scalar curvature. 
We denote $\lambda(A_g)=(\lambda_1(A_g),\dots,\lambda_n(A_g))$ as
the eigenvalues of $A_g$ with respect to $g$. Its $k$-th elementary symmetric polynomial is denoted by $\sigma_k(\lambda(A_g))\equiv\sigma_k(A_g)$, $1\leq k\leq n$, where 
\begin{equation*}
    \sigma_k(\lambda)\coloneqq\sum_{1\leq i_1<\dots<i_k\leq n}\lambda_{i_1}\cdots\lambda_{i_k}.
\end{equation*}

For $1\leq k\leq n/2$, the $\mathcal{B}^k_g$ curvature on the boundary $\p M$ was introduced by Chang and Chen in \cite{Chang-Chen, S.Chen-GAFA}:
\begin{equation*}
        \mathcal{B}^k_g\coloneqq \sum_{s=0}^{k-1} C(k,s) \sigma_{2k-s-1,s}(A_g^T,\sff_g),
\end{equation*}   
where
$A_g^T$ is the tangential part of the Schouten tensor, $\sff_g$ is the second fundamental form with respect to the unit inner normal $\vec{n}_g$, $\sigma_{i,j}$'s are the mixed elementary symmetric polynomials defined as in \cite{Reilly}, and $\textstyle C(k,s)= \frac{(2k-s-1)!(n-2k+s)!}{(n-k)!(2k-2s-1)!!s!}$.

When $k=1$, $\cB^1_g$ is the mean curvature $h_g\coloneqq \frac{1}{n-1}\tr_g \sff_g$ with respect to $\vec{n}_g$ (the boundary of a Euclidean ball has positive mean curvature).

When $\p M$ is umbilic, the $\mathcal{B}^k_g$ curvature takes the following form:
\begin{equation*}
    \mathcal{B}^k_g = \sum_{s=0}^{k-1} \alpha(k,s) \sigma_s(A_g^T) h_g^{2k-2s-1}, 
\end{equation*}
where we adopt the convention that $\sigma_0(A_g^T)=1$, and $\textstyle\alpha(k,s)= \frac{(n-1-s)!}{(n-k)!(2k-2s-1)!!}$.

The $\cB^k_g$ curvature arises naturally from the Chern--Gauss--Bonnet formula: When $(M^n,g)$ is locally conformally flat with $n=2k$, 
\begin{equation*}
    \frac{(2\pi)^k}{k!}\chi(M,\p M)=\int_M \sigma_k(A_g) dv_g + \oint_{\p M} \cB^k_g d\sigma_g,
\end{equation*}
where $\chi(M,\p M)$ is the Euler characteristic;  see \cite{Chang-Chen,S.Chen-GAFA}.

\begin{question}\label{quest}
    For $2 \leq l\leq k\leq n/2$ and constant $c\in \bR$, assume that $A_g\in\Gamma_k$ on $M$, i.e. $\sigma_i(A_g)>0$ for $i=1,\dots,k$. Does there exist a function $v$ on $M^n$ such that the conformal metric $g_v\coloneqq e^{2v}g$ satisfies
 \begin{equation*}
        \sigma_k(A_{g_v})=1 \quad \text{on}~M^n,
\end{equation*}
and the boundary $\cB_{g_v}^l$ curvature satisfies 
\begin{equation*}
         \cB_{g_v}^l=c\quad \text{on}~\p M\text{?}
\end{equation*}
\end{question}

The above question in the case when $k=l$ and $c=0$ was proposed by Chang and Chen in \cite{Chang-Chen}. Related problems on $\sigma_k(A_g)= 0$ in $M$ were studied by Case, Moreira, and Wang in \cite{Case-Wang,Case-Moreira-Wang}. See also \cite{Case-Wang-Jmathstudy} for progresses towards a sharp Sobolev trace inequality in terms of the $\sigma_k$ and $\cB_g^k$ curvatures. The boundary $B^k_{g_v}$ curvature involves both first and second-order derivatives of $v$ in a fully nonlinear form, which makes it challenging to manage. This level of complexity is rare in the literature.

Question \ref{quest} has the following variational formulation: When either $k=2$ or $(M^n,g)$ is locally conformally flat, equation $\sigma_k(A_{g_v}) = c_1$ on $M^n$ with $\cB_{g_v}^k=c_2$ on $\p M$, where $c_1,c_2\in\bR$, is the Euler-Lagrange equation of the functional 
\begin{align*}
    {\mathcal{I}}_k[v] = \mathcal{F}_k[v] - \frac{1}{n} \int_M c_1 \, dv_{g_v} - \frac{1}{n-1}\oint_{\p M} c_2 \, d \sigma_{g_v},
\end{align*}
where
\begin{equation*}
    \mathcal{F}_k [v] 
    :=
    \begin{cases}
        {(n-2k)^{-1}} \left[ \int_M \sigma_k (A_{g_v}) \, d v_{g_v} + \oint_{\p M} \cB^k_{g_v} \, d\sigma_{g_v} \right] \,\,&\text{if}\,\,n\neq 2k,\\
        \int_0^1 \left[ \int_M v \sigma_k(A_{g_{sv}})\,dv_{g_{sv}} + \oint_{\p M} v \cB^k_{g_{sv}}\,d\sigma_{g_{sv}} \right]\,ds \,\, &\text{if} \,\, n = 2k.
    \end{cases}
\end{equation*}
See \cite{viaclovsky-duke00,S.Chen-GAFA, Case-Wang,brendle-viaclovsky} for details.

When $l=1$, $\cB^1_{g}$ is the mean curvature as mentioned earlier. For $k=l=1$, Question \ref{quest} is the boundary Yamabe problem.
For $l=1$ and $k\geq 2$, Question \ref{quest} was proposed in \cite{Li-Li_JEMS}; see \cite{sophie-cvpde, chen-wei-100, He-Sheng, Jiang-Trudinger-oblique-1,Jiang-Trudiunger-2021, Jin-07, Jin-Li-Li, Li-Li_JEMS, Li2009, Li-Nguyen-umblic, Li-Nguyen-counterexample} and the references therein for studies on this question.

\smallskip

For a conformal metric $g_v\coloneqq e^{2v} g$ on $M^n$, it is known that 
\begin{align*}
A_{g_v}= -\nabla^2 v +  d v\otimes d v -\frac{1}{2} |\nabla v|^2 g + A_g, \\ \sff_{g_v}= e^{v}(\sff_g -\frac{\p v}{\p \vec{n}_g}g),\quad h_{g_v}=e^{-v}(h_g-\frac{\p v}{\p \vec{n}_g}),
\end{align*}
where all covariant derivatives and norms on the above are with respect to $g$, and $\vec{n}_g$ denotes, as before, the unit inner normal on $\p M$.

Denote  $\hrn\coloneqq \{(x_1,\dots,x_n)\in\bR^n\mid x_n>0\}$ as the Euclidean half space.

When $\bar{g} =  |dx|^2$ is the Euclidean metric on $\overline\hrn$, 
\begin{equation*}
    A_{\bar{g}_{v} } =e^{2 v} \av_{ij} dx^i dx^j, \quad \mathcal{B}^k_{\bar{g}_v}=\cB^k[v],
\end{equation*}
where
$\av$ is the conformal Hessian of $v$ defined by
\begin{equation*}
    \av\coloneqq e^{-2v}(-\nabla^2 v+\nabla v\otimes\nabla v-\frac{1}{2}|\nabla v|^2 I),
\end{equation*}
and $\cB^k[v]$ is defined by \begin{equation*}
    \cB^k[v]\coloneqq - \sum_{s=0}^{k-1}  \alpha(k,s) \sigma_s(A[v]^T) (e^{-v}\frac{\p v}{\p x_n})^{2k-2s-1}.
\end{equation*} 
Here $A[v]^T$ denotes the tangential part of $\av$.

 A map $\varphi:~\bR^n\cup \{\infty\} \rightarrow \bR^n \cup\{\infty\}$ is called a M\"obius transformation if it is a finite composition of translations, dilations, and inversions. 
 For a function $v$, define
\begin{equation*}
    v^{\varphi} := v \circ \varphi + \frac{1}{n}\log |J_\varphi|,
\end{equation*}
where $J_\varphi$ is the Jacobian of $\varphi$. 
It is known that $\av$ has the following M\"obius invariance
\begin{equation*}
    \lambda(A[v^\varphi])=\lambda(\av)\circ \varphi,
\end{equation*}
where $\lambda(M)$ denotes the eigenvalues of the symmetric matrix $M$ modulo permutations.
When $\varphi(\hrn)=\hrn$, it is known that for every $k$,
\begin{equation*}
    \cB^k[v^\varphi]=\cB^k[v].
\end{equation*}

\subsection{Liouville-type theorem}

For integers $l,k$, and $n$ satisfying $1\leq l\leq k\leq n/2$, consider  
\begin{equation}\label{equ-liouville-bk}
    \begin{cases}
        \sigma_k(\av)=e^{-pv},~\av\in\Gamma_k\quad \text{on}~\overline\hrn,\\
        \cB^l[v]=c\quad \text{on}~\p\hrn,
    \end{cases}
\end{equation}
where $p\geq 0$ and $c\in\bR$ are some constants, and $\Gamma_k\coloneqq\{M\in \cS^{n\times n}\mid \sigma_i(M)>0,~1\leq i\leq k\}$.

Let
\begin{equation}\label{half-space-bubble}
        V(x)\equiv \log \left(\frac{c_{n,k}a}{1+a^2|x-\bar{x}|^2}\right),
    \end{equation}
where $a>0$, $\bar{x}\in\bR^n$, and $c_{n,k}=\sqrt{2}{\binom{n}{k}}^{{1}/{(2k)}}$. It is known that $A[V]\equiv 2 c_{n,k}^{-2}I$, and thus $V$ is a solution of 
\begin{equation}\label{interior-equ-sigma_k}
    \sigma_k(\av)=1,~ \av\in\Gamma_k\quad \text{in}~ \bR^n.
\end{equation}

For $V$ of the form \eqref{half-space-bubble} with 
$\bar x=(\bar x', \bar x_n)$ satisfying 
\begin{equation}\label{bubble-boundary-condition}
    \sum_{s=0}^{l-1}\alpha(l,s)\binom{n-1}{s}(\sqrt{2}a\bar{x}_n)^{2l-2s-1}=c{\binom{n}{k}}^{{(2l-1)}/{(2k)}},
\end{equation}
we have $e^{-V}{\p_{x_n} V}=2c_{n,k}^{-1}a\bar{x}_n$ and $\cB^l[V]=c$ on $\p\hrn$.

\medskip
Wei proved that 

\medskip

\noindent {\bf Theorem A.}\ (\cite{wei2024})\
{\it For $2\leq k\leq n/2$, let $v\in C^2(\overline\hrn)$ satisfy \eqref{equ-liouville-bk} with $l=k$, $p=0$, and $c\geq 0$. Assume that 
$\lim_{|x|\to\infty}(v(x)+2\log|x|)$ exists. Then $v$ is of the form \eqref{half-space-bubble} with $a$ 
and $\bar x$ satisfying \eqref{bubble-boundary-condition}.}
\medskip

In this paper, we prove the following Liouville-type theorem.

\begin{theorem}\label{bk-liouville}
    For $2 \leq l\leq k\leq n/2$, let $v\in C^2(\overline\hrn)$ satisfy \eqref{equ-liouville-bk} for some constants $c,p\geq 0$. Then $p=0$ and  v is of the form \eqref{half-space-bubble} with $a$ and $\bar x$ satisfying \eqref{bubble-boundary-condition}.
\end{theorem}

\begin{remark}\label{remark-1.1}
    The above theorem still holds if $\sigma_k$ is replaced by some symmetric function\footnote{By symmetric function, we mean that $f$ is invariant under permutation of $\lambda_i$'s.} $f\in C^{0,1}_{loc}(\Gamma_k)$ which satisfies $\p_{\lambda_i}f\geq c(K)>0$,~$\forall i$ a.e. on any compact subset $K$ of $\Gamma_k$.  See the end of \S \ref{liouville-sec} for reasons.
\end{remark}

When $l=1$, $\cB^1[v]=-e^{-v} {\p_{x_n} v}$ is the mean curvature of the conformal metric $e^{2v}|dx|^2$ on $\p\hrn$. It was previously known that for $n\geq 3$, $l=1$, and $1\leq k\leq n$, all solutions of \eqref{equ-liouville-bk} are of the form \eqref{half-space-bubble} with $a$ and $\bar x$ satisfying \eqref{bubble-boundary-condition};
see \cite{Li-Zhu, Li-Li_JEMS}.
For more general equations $f(\av)=1$, see \cite{Li-Li_JEMS, CLL-2}.
For classification results for $n\geq 3$ and $k=l=1$ under some additional assumptions on solutions, see \cite{escobar_cpam}.

Liouville-type theorem for equation \eqref{interior-equ-sigma_k} was previously known: For $n\geq 3$ and $1\leq k\leq n$, all solutions of \eqref{interior-equ-sigma_k} are of the form \eqref{half-space-bubble}; see \cite{CGS,MR1945280,Li-Li_CPAM03}. For more general equations $f(\av)=1$, see \cite{LiLi_acta,Li2021ALT, CLL-1}. For classification results under some additional assumptions on solutions, see
\cite{aubin-bubble, talenti, viaclovsky-duke00, GNN,obata, Chen-Li}.

\smallskip
In Theorem \ref{bk-liouville}, no assumptions are made on the solution 
$v$ near infinity. A key contribution of our proof is a sufficient understanding of the behavior of solutions near an isolated boundary singularity. For instance, let $v$ satisfy \eqref{equ-liouville-bk} for $p=0$. By the M\"obius invariance of \eqref{equ-liouville-bk}, we have 
\begin{equation*}
    \begin{cases}
        \sigma_k(A[v^{\varphi_{0,1}}])=1,~A[v^{\varphi_{0,1}}]\in\Gamma_k\quad \text{on}~\overline\hrn\setminus\{0\},\\
        \cB^l[v^{\varphi_{0,1}}]=c\quad \text{on}~\p\hrn\setminus\{0\},    \end{cases}
\end{equation*}
where $\varphi_{0,1}(x)=x/|x|^2$ is the inversion with respect to the unit sphere $\p B_1(0)$ in $\bR^n$. Assume that $v^{\varphi_{0,1}}>v$ in $B_1(0)\cap\hrn$, we prove
\begin{equation*}
\liminf_{x\to 0}(v^{\varphi_{0,1}}-v)(x)>0.
\end{equation*}
See Proposition \ref{mp-singular}, where a more general result is proved. 

\subsection{Local gradient estimate}
For integers $l,k$, and $n$ satisfying $2\leq l\leq k\leq n/2$ and some constant $c\geq 0$, consider
\begin{equation}\label{equ-lgest-bk-euclidean}
    \begin{cases}
        \sigma_k^{1/k}(\av)=1,~\av\in\Gamma_k\quad \text{in}~B_1^+\cup\p' B_1^+,\\
        \cB^l[v]=c\quad \text{on}~\p' B_1^+,
    \end{cases}
\end{equation}
where $B_r$ denotes the ball $B_r(0)$ in $\bR^n$ with center at $0$ and radius $r$,  $B_r^+\coloneqq B_r\cap\hrn$, and $\p' B_r^+$ denotes the interior of $\p B_r^+\cap\p\hrn$.
\begin{theorem}\label{lgest-bk-eucildean}
    For $2\leq l\leq k\leq n/2$ and $c\geq 0$, let $v\in C^3(B_1^+\cup\p' B_1^+)$ be a solution of \eqref{equ-lgest-bk-euclidean}. Then 
    \begin{equation}\label{lgest-estimate}
        |\nabla v|\leq C\quad\text{in}~B^+_{1/2}\cup \p' B^+_{1/2},
    \end{equation}
    where $C$ is a constant depending only on $l,k,n$, and an upper bound of $c$ and $\sup_{B_1^+}v$.
\end{theorem}

\begin{remark}\label{remark-1.2}
    The above theorem still holds if $\sigma_k$ is replaced by some symmetric, homogeneous of degree $1$ function  $f\in C^0(\overline{\Gamma_k})\cap C^{1}(\Gamma_k)$ which satisfies $f|_{\Gamma_k}>0$, $f|_{\p\Gamma_k}=0$, $\sum_i\p_{\lambda_i}f|_{\Gamma_k}\geq \delta>0$, and $\p_{\lambda_i} f|_{\Gamma_k}>0$ for all $i$. See the end of \S \ref{lgest-sec} for reasons.
\end{remark}

\begin{remark}
After communicating our work to Wei Wei, we 
received her  independent work where she 
also obtained Theorem 1.2 and the extension stated in Remark 1.2.
\end{remark}

The above local gradient estimate only assumes an upper bound of solutions. Such estimates will be useful in answering Question \ref{quest}.

When $l=1$, $\cB^1[v]=-e^{-v} {\p_{x_n} v}$ is the mean curvature as mentioned earlier. In this case, the local gradient estimate of equation \eqref{equ-lgest-bk-euclidean} was proved by Li in \cite{Li2009}. In fact, more general equation $f(A_{g_v})=h$ with $\Gamma\subset\Gamma_1$ was treated there. In our forthcoming paper, we derive necessary and sufficient conditions for the validity of such local gradient estimates.

In the interior case, the local gradient estimate for equation $\sigma_k(A_{g_v})=h,~A_{g_v}\in\Gamma_k$, for $2\leq k\leq n$ was proved by Guan--Wang in \cite{guan-wang_imrn}. For general equation $f(A_{g_v})=h$ with $\Gamma\subset\Gamma_1$, the local gradient estimate was proved by Li in \cite{Li2009}. In our recent work \cite{CLL-1}, we obtained necessary and sufficient conditions for the validity of such estimates.

Using the strategy and Liouville-type theorems established in \cite{Li2009},
the proof of Theorem \ref{lgest-bk-eucildean} is reduced to obtaining estimate \eqref{lgest-estimate} assuming both upper and lower bounds of solution $v$.
This estimate is achieved through Bernstein-type arguments.

\subsection{Notations}  For a set $\Omega\subset\bR^n$, we denote $\Omega^+\coloneqq \Omega\cap\hrn$, $\p'\Omega\coloneqq$ the interior of the set $\p \Omega^+\cap \p \bR^n_+$, and $\p''\Omega\coloneqq \p\Omega^+\setminus\p'\Omega$. For simplicity, we denote $B$ as the unit ball $B_1(0)$ in $\bR^n$. 
We use a subscript $T$ to denote the tangential part of a differential operator, e.g. $\nabla_T$, $\nabla^2_T$, and $\Delta_T$. We use a superscript $T$ to denote the tangential part of a tensor, e.g. $A_g^T$ and $\av^T$. 
For simplicity, we write $\lambda(\av)\in\Gamma_k$ as $\av\in\Gamma_k$, and $\sigma_k(\lambda(\av))$ as $\sigma_k(\av)$.

\subsection{Organization of the paper}
\S \ref{pre-sec} is a preliminary section on maximum principles for regular solutions. In \S \ref{mp-singular-sec}, we prove a maximum principle for solutions with isolated boundary singularities. Theorem \ref{bk-liouville} is proved in \S \ref{liouville-sec}. In \S \ref{lgest-sec}, we prove Theorem \ref{lgest-bk-eucildean}.

\section{Preliminaries}\label{pre-sec}
This section is devoted to the following two propositions on the maximum principles for solutions without singularities. 

\begin{proposition}\label{mp-regular}
    For $2\leq l\leq k\leq n/2$, let $u,v\in C^2(\overline{B^+})$ satisfy $A[u],\av\in\Gamma_k$ in $\overline{B^+}$, and
    \begin{equation}\label{equ-mp-regular}
        \begin{cases}
            \sigma_k(A[u]) +\psi(x,u,\nabla u)\geq \sigma_k(\av) +\psi(x,v,\nabla v)\quad \text{in}~\overline{B^+},\\
            \cB^l[u]\geq \cB^l[v]\geq 0 \quad \text{on}~\p' B^+,
        \end{cases}
    \end{equation}
    where $\psi\in C^{0,1}_{loc}(\overline{B^+}\times\bR\times\bR^n)$.
    Assume that $u>v$ in $B^+$, then $u(0)>v(0)$.
\end{proposition} 

\begin{proposition}\label{lem-corner-hopf}
    For $2\leq l\leq k\leq n/2$, let $u$ and $v$ be $C^2$ functions in the closed cube $Q\coloneqq \{x\in\bR^n\mid 0\leq x_1, x_n\leq 1,~-1/2\leq x_i\leq 1/2,~2\leq i\leq n-1\}$ satisfying $A[u],\av\in\Gamma_k$ in $Q$. Suppose that 
    \begin{equation}\label{equ-corner-hopf}
        \begin{cases}
            \sigma_k(A[u]) +\psi(x,u,\nabla u)\geq \sigma_k(\av) +\psi(x,v,\nabla v)\quad \text{in}~Q,\\
            \cB^l[u]\geq \cB^l[v]\geq 0 \quad \text{on}~\p' Q,\\
    u>v~\text{in}~Q^+,~u(0)=v(0),        \end{cases}
    \end{equation}
    where $\psi\in C^{0,1}_{loc}( Q\times\bR\times\bR^n)$.
    Then $\frac{\p }{\p x_1}(u-v)(0)>0$.
\end{proposition} 

The proofs of the above two propositions follow from the same arguments in \cite{wei2024}, where the case when $k=l$ and $\psi=0$ were treated. 

\smallskip

The proofs of Proposition \ref{mp-regular} and \ref{lem-corner-hopf} make use of the following lemma from \cite{wei2024}, see Lemma 3 there; see also \cite[Lemma 2.2]{Case-Wang} .
\begin{lemma}\label{linearization-bk}
    For $2\leq l\leq n/2$, let $\Omega$ be an open set in $\bR^n$ and $u,v\in C^2(\Omega^+\cup\p'\Omega^+)$. Denote $w_t\coloneqq v+ t(u-v)$, $0\leq t\leq 1$. Then the following holds on $\p'\Omega^+$:
    \begin{equation}\label{240602-1532}
       \cB^l[u]-\cB^l[v]=\int_0^1\frac{d}{dt}\cB^l[w_t]dt = \cL(u-v),
    \end{equation}
    where $\cL$ is the linear operator
    \begin{equation*}
            \cL\coloneqq -a_{\alpha\beta}\frac{\p^2}{\p x_\alpha\p x_\beta}+b_\beta \frac{\p}{\p x_\beta} -c_n \frac{\p}{\p x_n} -d_0  \end{equation*}
with the coefficients defined as follows:
    \begin{gather*}
        a_{\alpha\beta}\coloneqq \int_{0}^1\sum_{s=1}^{l-1}\alpha(l,s)\frac{\p\sigma_s}{\p M_{\alpha\beta}}(A[w_t]^T)h_{\bar{g}_{w_t}}^{2l-2s-1} e^{-2 w_t}dt,\\
        b_\beta\coloneqq \int_{0}^1\sum_{s=1}^{l-1}\alpha(l,s) \left( 2 \frac{\p\sigma_s}{\p M_{\alpha\beta}}(A[w_t]^T)\frac{\p w_t}{\p x_\alpha}- \frac{\p\sigma_s}{\p M_{\beta\beta}}(A[w_t]^T)\frac{\p w_t}{\p x_\beta}        \right)
        h_{\bar{g}_{w_t}}^{2l-2s-1} e^{-2 w_t}dt,\\
        c_n\coloneqq \int_0^1 \sigma_{l-1}(A[w_t]^T) e^{-w_t}dt,\quad d_0\coloneqq (2l-1)\int_0^1 \sum_{s=0}^{l-1} \alpha(l,s)\sigma_s(A[w_t]^T) h_{\bar{g}_{w_t}}^{2l-2s-1} dt.   \end{gather*}
        The above formulas adopt the summation convention on indices $\alpha,\beta=1,\dots,n-1$.
\end{lemma}
 For reader's convenience, we include a proof of Lemma \ref{linearization-bk} in Appendix \ref{compute-appendix}. 
 
The proofs of Proposition \ref{mp-regular} and \ref{lem-corner-hopf} also use the following lemma, which gives the ellipticity of $\cL$.

\begin{lemma}\label{convexity-lem} 
For $n\geq 2$, let $u, v$ be functions in $\bR^n$ which are second order differentiable at $x_0$. Denote $w_t=v +t(u-v)$, $0\leq t\leq 1$, the following holds at $x_0$. 
\begin{equation}\label{convexity-identity}
      A[w_t]=te^{2(1-t)(u-v)}A[u]+(1-t)e^{-2t(u-v)}\av+ \cC,
\end{equation}      
where 
\begin{equation*}
\cC=t(1-t) e^{-2w_t}\left(\frac{1}{2}|\nabla(u-v)|^2 I-\nabla(u-v)\otimes\nabla(u-v)\right).
\end{equation*}
As a consequence, assume that $A[u], A[v]\in \Gamma_k$ at $x_0$, for some $1\leq k\leq n/2$, then $A[w_t]\in \Gamma_k$ at $x_0$.  
\end{lemma}

The above Lemma when $u=0$ was proved in \cite[Page 34]{He-Xu-Zhang}. See Appendix \ref{convex-app} for discussions for general cone $\Gamma$.

\begin{proof}
    The algebraic identity \eqref{convexity-identity} can be verified by a direct computation.
    
    Next we prove the last assertion in Lemma \ref{convexity-lem}. Since $A[u],\av\in\Gamma_k$, by the convexity and the cone property of $\Gamma_k$, we have 
    \begin{equation}\label{240923-2055}
        te^{2(1-t)(u-v)}A[u]+(1-t)e^{-2t(u-v)}\av\in \Gamma_k.
    \end{equation}
    
    On the other hand, it is easy to see that
    the eigenvalues of $\cC$ take the following form:
    \begin{equation*}
        \lambda(\cC)=\frac{1}{2} t(1-t)e^{-2 w_t}|\nabla(u-v)|^2(1,\cdots,1,-1).    \end{equation*} 
    Hence, by $k\leq n/2$, we have $\cC\in\overline{\Gamma_k}$. Combining this with \eqref{convexity-identity} and \eqref{240923-2055}, again by the convexity and the cone property of $\Gamma_k$, we have $A[w_t]\in\Gamma_k$ at $x_0$.
\end{proof}

\begin{proof}[Proof of Proposition \ref{mp-regular}]
   Suppose the contrary that $u(0)=v(0)$. Then, we have $\nabla_T u=\nabla_T v$, $\frac{\p u}{\p x_n}\geq \frac{\p v}{\p x_n}$ and $\nabla_T^2 u\geq \nabla_T^2 v$ at $0$. Let $w_t\coloneqq t u+(1-t) v$, $0\leq t\leq 1$, by Lemma \ref{linearization-bk}, we have 
   \begin{equation}\label{linearization}
     \cB^l[u]-\cB^l[v]
     =-a_{\alpha\beta}\frac{\p^2(u-v)}{\p x_\alpha\p x_\beta}-c_n\frac{\p (u-v)}{\p x_n} \quad \text{at}~0,
   \end{equation}
   where $a_{\alpha\beta}$ and $c_n$ are those in Lemma \ref{linearization-bk}.

 Since $A[u],\av\in\Gamma_k$ at $0$ and $k\leq n/2$, by Lemma \ref{convexity-lem}, we have $A[w_t](0)\in\Gamma_k$. Therefore, $A[w_t]^T(0)\in\Gamma_{k-1}$. Since $l\leq k$, we have
   \begin{equation}\label{admiss-conseq}
    A[w_t]^T(0)\in \Gamma_{s},\quad  s=1,\dots,l-1.   
   \end{equation}

   Using $\cB^l[u],\cB^l[v]\geq 0$, we have $h_{\bar{g}_u},h_{\bar{g}_v}\geq 0$ at $0$. Therefore, $h_{\bar{g}_{w_t}}(0)\geq 0$. Combining this with \eqref{admiss-conseq}, we have $(a_{\alpha\beta})\geq 0$ and $c_n>0$ at $0$.
    Moreover, by the Hopf Lemma, we have $\tfrac{\p(u-v)}{\p x_n}(0)>0$. Combining these and formula \eqref{linearization}, we have $\cB^l[u]-\cB^l[v]<0$ at $0$. On the other hand, the boundary condition of \eqref{equ-mp-regular} implies that $\cB^l[u]-\cB^l[v]\geq 0$ at $0$. A contradiction.
\end{proof}

The proof of Proposition \ref{lem-corner-hopf} also makes use of the following lemma, which is similar to \cite[Lemma 10.1]{MR2001065} and \cite[Lemma 14]{wei2024}. For completeness, we include a proof in the Appendix \ref{compute-appendix}.
\begin{lemma}\label{corner-hopf-linear}
        For $n\geq 3$, let $w$ be a $C^2$ function in the cube $Q\coloneqq \{x\in\bR^n\mid 0\leq x_1, x_n\leq 1,~-1/2\leq x_i\leq 1/2,~i=2,\dots,n-1\}$. Assume that
    \begin{equation*}
        \begin{cases}
            -A_{ij}w_{ij} +B_i w_i\geq -A w\quad \text{in}~Q,\\
            -a_{\alpha\beta} w_{\alpha\beta} + b_\beta w_{\beta}-c_n w_n \geq -A w \quad \text{on}~Q\cap\p\hrn,\\
    w>0~\text{in}~Q\setminus\{0\},~w(0)=0,        \end{cases}
    \end{equation*}
    where  $A_{ij},~B_i\in L^\infty(Q)$ for $i,j=1,\dots,n$, $a_{\alpha\beta}$, $b_\beta\in L^\infty(Q\cap\p\hrn)$ for $\alpha,\beta=1,\dots,n-1$, $c_n$ is some nonnegative function on $Q\cap\p\hrn$,
$A$ is some positive constant, and 
there exists some constant $\lambda>0$ such that $(A_{ij})\geq 0$, $A_{11}\geq \lambda$ in $Q$, and $(a_{\alpha\beta})\geq 0$, $a_{11}\geq \lambda$ on $Q\cap\p\hrn$.
Then we have $\frac{\p w}{\p x_1}(0)>0$.
\end{lemma}

\begin{remark}\label{accomodate-zero}
    The assumption ``$a_{11}\geq \lambda >0$ on $Q\cap\p\hrn$" can be replaced by ``$a_{11}(x)=o(1)$, $b_1(x)=o(1)$ as $x\to 0$, and $c_n\geq \lambda>0$ on $Q\cap\p\hrn$". This can be seen easily from the proof.   
\end{remark}

\begin{proof}[Proof of Proposition \ref{lem-corner-hopf}]
    By Proposition \ref{mp-regular}, $u>v$ on $\p' Q^+$. Let $w=u-v$, then $w>0$ in $2^{-1}Q\setminus \{0\}$ and $w(0)=0$.    
    By \eqref{equ-corner-hopf}, we have, denote $w_t=v+t(u-v)$, $0\leq t\leq 1$,     \begin{equation*}
        \begin{cases}
           -A_{ij} w_{ij}+ B_i w_i= D w\geq -A w\quad \text{in}~2^{-1}Q, \\
           -a_{\alpha\beta} w_{\alpha\beta} +b_\beta w_\beta -c_n w_n =d_0 w \geq -A w\quad \text{on}~2^{-1}Q\cap \p\hrn,
        \end{cases}
    \end{equation*}
where $A_{ij}=\int_0^1\frac{\p\sigma_s}{\p M_{ij}}(A[w_t])e^{-2 w_t}dt$, $B_i$, $D\in L^\infty(2^{-1}Q)$, coefficients $a_{\alpha\beta}$, $b_\beta$, $c_n$, and $d_0$ are given in Lemma \ref{linearization-bk}, and $A$ is some positive constant. 

Since $A[u],A[v]\in \Gamma_k$ in $2^{-1}Q$, by Lemma \ref{convexity-lem} and $k\leq n/2$, we have $A[w_t]\in\Gamma_k$ in $2^{-1}Q$, $0\leq t\leq 1$. Therefore, $A[w_t]\in\Gamma_{k-1}$ in $2^{-1}Q$, $0\leq t\leq 1$. It follows that there exists some constant $\lambda>0$ such that $(A_{ij})\geq \lambda I_n$ in $2^{-1}Q$, $(a_{\alpha\beta})\geq 0$, and $c_n\geq\lambda$ on $2^{-1}Q\cap\p\hrn$. By Lemma \ref{corner-hopf-linear} and Remark \ref{accomodate-zero}, we have $\tfrac{\p w}{\p x_1}(0)>0$.  
\end{proof}

\section{Maximum principle for singular solutions}\label{mp-singular-sec}
The crucial new ingredient in the proof of Theorem \ref{bk-liouville} is the following proposition.
\begin{proposition}\label{mp-singular}
    For $2\leq l\leq k\leq n/2$ and $p\in\bR$, let $u\in C^2(\overline{B^+}\setminus\{0\})$ and $v\in C^2(\overline{B^+})$ satisfy $A[u]\in\Gamma_k$ in $\overline{B^+}\setminus\{0\}$ and $\av\in\Gamma_k$ in $\overline{B^+}$. Assume that \begin{align}\label{equ-mp-singular}
        \begin{cases}
     e^{pu}\cdot \sigma_k(A[u])\geq e^{pv}\cdot \sigma_k(\av) \quad &\text{in}~B^+,\\
            \cB^l[u]\geq \cB^l[v]\geq 0 \quad &\text{on}~\p' B^+\setminus\{0\},
        \end{cases}
    \end{align}
    and that $u>v$ in $B^+$. Then $\liminf_{x\to 0}(u-v)(x)>0$.
\end{proposition}

We make use of the following lemma in the proof of Proposition \ref{mp-singular}.

\begin{lemma}\label{vansihgrad}
    Let $n\geq 2$ and $v$ be a $C^1$ function defined in a neighborhood of $0\in \overline{\hrn}$ with $q \coloneqq \nabla_T v(0) \neq 0$.
Then there exists a M\"obius transformation $\psi$ such that 
\begin{equation*}
    \psi(0) = 0,~ \psi(\hrn)=\hrn,~ \text{and}~ \nabla_T v^\psi (0) = 0.
\end{equation*}
Moreover, all such $\psi$ can be written as
\begin{equation} \label{240201-1609}
    \psi(x) = O \left( \frac{\lambda^2 (x-\bar{x}) }{|x - \bar{x}|^2} + \frac{\lambda^2 \bar{x}}{|\bar{x}|^2} \right),
\end{equation}
where $\lambda \in \bR\setminus\{0\}$, $O=\operatorname{diag}\{O',1 \}$, $O'\in O(n-1)$ and $\bar{x} = \frac{\lambda^2}{2} O^T
     ( q^T, 0)^T$.
\end{lemma}

\begin{proof}
    First, from \cite[Theorem~3.5.1]{MR1393195}, all M\"obius tranforms mapping both zero to zero and $\hrn$ to $\hrn$ can be written as either \eqref{240201-1609} or $\psi(x) = \lambda^2 O x$ with $O=\operatorname{diag}\{O',1\}$, $O'\in O(n-1)$ and ${\bar x}_n=0$.  Clearly, the second possibility can be ruled out as $\nabla_T v^\psi(0) = \lambda^2 O^T q \neq 0$. Now, for $\psi$ in the form of \eqref{240201-1609}, a direct computation shows
\begin{equation*}
    \nabla v^\psi (0) 
    = 
    2 \frac{\bar{x}}{|\bar{x}|^2} 
    + 
    \frac{\lambda^2}{|\bar{x}|^2} \left( I_n - 2 \frac{\bar{x}}{|\bar{x}|} \otimes \frac{\bar{x}}{|\bar{x}|} \right) O^T \nabla v(0).
\end{equation*}
Matching $\nabla_T v^\psi (0) = 0$, we obtain the desired formula.  
\end{proof}

\begin{proof}[Proof of Proposition \ref{mp-singular}]
    Suppose the contrary that $v(0)=u(0)\coloneqq \liminf_{x\to 0}u(x)$. We will derive a contradiction. 

\smallskip

We only need to address the case where $\nabla_Tv(0)=0$, since otherwise
by Lemma \ref{vansihgrad}, there exists a M\"obius transformation $\psi$ in the form of \eqref{240201-1609} such that $\psi(0) = 0$, $\psi(\hrn)=\hrn$, and $\nabla_T v^\psi (0) = 0$.
By \eqref{equ-mp-singular} and its M\"obius invariance, a direct computation gives, for some $\delta>0$,  $A[u^\psi]\in\Gamma_k$ in $\overline{B_\delta^+}\setminus\{0\}$, $A[v^\psi]\in\Gamma_k$ in $\overline{B_\delta^+}$, 
$\cB^l[u^\psi]\geq \cB^l[v^\psi]\geq 0$ on $\p' B_\delta^+\setminus\{0\}$, and
\begin{gather*}
    e^{p u^\psi} \sigma_k(A[u^\psi]) -
    e^{p v^\psi} \sigma_k(A[v^\psi])  \\
    =  \lambda^{2p}|\cdot-\bar{x}|^{-2p}    \left( e^{p u\circ\psi} \sigma_k(A[u\circ \psi]) -
    e^{p v\circ\psi} \sigma_k(A[v\circ\psi]) \right)\geq 0 \quad \text{in}~B_\delta^+.
\end{gather*}
Since $\nabla_T v^\psi(0)=0$,  we apply the conclusion of the former case to $u^\psi$ and $v^\psi$ to obtain $\liminf_{x\to 0}(u^\psi-v^\psi)(x)>0$, which implies $\liminf_{x\to 0}(u-v)(x)>0$.
\smallskip

In the remaining of the proof, we assume $\nabla_T v(0)=0$.

The following lemma should be read as part of the proof.
\begin{lemma}\label{key-lem}
    For $\delta>0$, there exists some $r_0>0$ such that $u(0)\coloneqq \liminf_{x\to 0}u(x)$ satisfies, for all $r\in (0,r_0)$, that
    \begin{equation*}
        u(0)\geq \inf_{x\in\p'' B_r^+}\left( u(x)-v_0(x) \right),
    \end{equation*}
    where 
    \begin{equation*}
         v_0(x)\coloneqq \left(\frac{\p v}{\p x_n}(0)+\delta\right)x_n + \frac{1}{2} \left\langle x',\nabla_T^2 v(0)\cdot x'\right\rangle + \left|\Delta_T v(0)\right|x_n^2.          \end{equation*}
\end{lemma}

\begin{proof}[Proof of Lemma \ref{key-lem}]
  Let
  \begin{equation*}
      w_\alpha(x)\coloneqq v_\alpha(x) +v_0(x),\quad x\in \overline{B^+}\setminus\{0\}, 
  \end{equation*}
where
$v_\alpha\coloneqq \alpha\log|x|$, and
$\alpha>0$ is some constant which will be sent to $0$.

The proof consists of two steps.

\textbf{Step 1:} Prove that there exists some $r_1>0$, independent of $\alpha$, such that \begin{equation}\label{240523-1305}
 A[w_\alpha]\in\bR^n\setminus\overline{\Gamma_1}\quad \text{on} ~~\overline{B^+_{r_1}}\setminus\{0\}.
 \end{equation}
       
    For a function $\varphi$, we denote $W[\varphi]\coloneqq e^{2\varphi}A[\varphi]=-\nabla^2\varphi+\nabla \varphi\otimes\nabla\varphi -\frac{1}{2}|\nabla \varphi|^2 I$. A direct computation yields
     \begin{equation*}
         W[w_\alpha]= W[v_0] + \left( W[v_\alpha] + E \right),
     \end{equation*}
     where 
     \begin{equation*}
         E=\nabla v_\alpha \otimes \nabla v_0 +\nabla v_0\otimes \nabla v_\alpha -(\nabla v_\alpha\cdot \nabla v_0) I= \alpha O(|x|^{-1}).
     \end{equation*}
     
      By a calculation, 
      \begin{equation*}
          W[v_\alpha](x)=\alpha |x|^{-2}\{ (-1-\frac{\alpha}{2})I +(2+\alpha)\frac{x}{|x|}\otimes \frac{x}{|x|}\},     \end{equation*}
     and thus we have
    $\tr(W[v_\alpha]+E)(x)=\alpha|x|^{-2}\left(-(n-2)(1+\alpha/2)+O(|x|)\right)$.  
    Since $\alpha>0$, there exists some $\hat{r}_1>0$, independent of $\alpha$, such that
    \begin{equation}\label{240919-2252}
      (W[v_\alpha]+E)(x)\in\bR^n\setminus\overline{\Gamma_1},\quad 0<|x|<\hat{r}_1.   \end{equation}

    A computation gives that 
    \begin{equation*}
    W[v_0]=-\operatorname{diag}\{\nabla_T^2 v(0), 2|\Delta_T v(0)|\}+\frac{1}{2}(\frac{\p v}{\p x_n}(0)+\delta)^2 \operatorname{diag}\{-I_{n-1},1\}+O(|x|).
    \end{equation*}
    Since $A[v](0)\in\Gamma_k$ and $k\geq 2$, we have $A[v]^T(0)\in \Gamma_1$, i.e. 
    \begin{equation*}
        \Delta_T v(0)<-\frac{n-3}{2}|\nabla_T v(0)|^2 -\frac{n-1}{2}(\frac{\p v}{\p x_n}(0))^2.    \end{equation*}
 The above implies that $\Delta_T v(0)<0$. Combining this with the expression of $W[v_0]$, we have
    \begin{align*}
    \tr(W[v_0]) &=-\Delta_T v(0)-2|\Delta_T v(0)|- \frac{n-2}{2}(\frac{\p v}{\p x_n}(0)+\delta)^2 +O(|x|)\\
    &\leq -|\Delta_T v(0)|+O(|x|).
    \end{align*}
    Hence, there exists $\widetilde{r}_1>0$, independent of $\alpha$, such that 
    \begin{equation}\label{240919-2258}
      W[v_0](x)\in\bR^n\setminus\overline{\Gamma_1},\quad 0<|x|<\widetilde{r}_1. 
    \end{equation}
    
Using \eqref{240919-2252}, \eqref{240919-2258}, and the convex cone property of $\bR^n\setminus\overline{\Gamma_1}$, we obtain \eqref{240523-1305} with $r_1=\min\{\hat{r}_1,\widetilde{r}_1\}$.   Step 1 is completed.

    \medskip

    For $0<r<\min\{1,r_1\}$, let 
    \begin{equation*}
        w_\alpha^{(r)}(x)\coloneqq w_\alpha(x)+\inf_{x\in \p'' B_r^+}\left(  
        u(x)-v_0(x)\right).
\end{equation*}

   \medskip
   
 \textbf{Step 2:} Prove that there exists some $r_0>0$, independent of $\alpha$, such that 
 \begin{equation}\label{order-claim}
      w_\alpha^{(r)}\leq u\quad \text{in}~~ \overline{B_r^+}\setminus\{0\},~ \forall 0<r<r_0,~ \forall \alpha>0. \end{equation}

   We argue by contradiction. If not, there exists two sequences of positive numbers $\{r_j\}$ and $\{\alpha_j\}$ with $r_j\to 0$ such that 
   \begin{equation*}
       c_j\coloneqq -\inf_{x\in\overline{B^+_{r_j}}\setminus\{0\}}(u-w_{\alpha_j}^{(r_j)})(x)>0.
   \end{equation*}
     It is clear that $w_\alpha^{(r)}(x)\to -\infty$ as $x\to 0$ and $w_\alpha^{(r)}<u$ on $\p'' B_r^+$. Hence, there exists some $x_j\in B_{r_j}^+\cup(\p' B_{r_j}^+\setminus\{0\})$ such that
   $c_j=-(u-w_{\alpha_j}^{(r_j)})(x_j)$. 
   
   Let 
   \begin{equation*}
          w_j\coloneqq w_{\alpha_j}^{(r_j)}-c_j.   \end{equation*}
 Then $w_j$ touches $u$ from below at $x_j$, i.e. $w_j\leq u$ in $\overline{B^+_{r_j}}\setminus\{0\}$ and $w_j=u$ at $x_j$.
   
   There are two possibilities:  $x_j\in B_{r_j}^+$ and $x_j\in \p' B_{r_j}^+\setminus\{0\}$.

   If $x_j\in B_{r_j}^+$, then $\nabla w_j(x_j)=\nabla u(x_j)$ and $\nabla^2 w_j(x_j)\leq \nabla^2 u(x_j)$, from which $A[w_j](x_j) \geq A[u](x_j)$. 
   Since $A[u](x_j) \in \Gamma_k$, we have $A[w_j](x_j)\in\Gamma_k\subset\Gamma_1$. On the other hand, by \eqref{240523-1305}, we have $A[w_j]\in\bR^n\setminus\overline{\Gamma_1}$ in $\overline{B_{r_j}^+}\setminus\{0\}$.   
   A contradiction. 

   We are left to rule out the possibility $x_j\in \p' B_{r_j}^+\setminus\{0\}$. In this case,
   \begin{equation}\label{derivative-test}
       \nabla^2_T u\geq \nabla^2_T w_j,~~ \nabla_T u=\nabla_T w_j,~~\text{and}~~\frac{\p u}{\p x_n}\geq \frac{\p w_j}{\p x_n}\quad \text{at}~x_j.
   \end{equation}
   
   Write
   \begin{equation}\label{240919-2326}
       W[u]^T=W_1[u]+W_2[u], \end{equation}
    where 
    \begin{equation*}
        W_1[u]\coloneqq -\nabla_T^2 u-\frac{1}{2}(\frac{\p u}{\p x_n})^2 I_{n-1},~ W_2[u]\coloneqq -\frac{1}{2}|\nabla_T u|^2 I_{n-1} +\nabla_T u\otimes\nabla_T u.    \end{equation*}
        
 Since $\lambda(W_2[u])=-2^{-1}|\nabla_T u|^2(1,\dots,1,-1)$, it is easy to verify that
$-W_2[u]\in \overline{\Gamma_{\lfloor \frac{n-1}{2} \rfloor}}$.
 On the other hand, $W[u](x_j)= e^{2u(x_j)} A[u](x_j)\in\Gamma_k$ implies that $W[u]^T(x_j)\in\Gamma_{k-1}$. Since $l\leq k\leq n/2$, we have $-W_2[u],W[u]^T\in \overline{\Gamma_s}$ at $x_j$ for $s=1,\dots,l-1$. Note that $W_1[u]=W[u]^T+(-W_2[u])$, we have
   \begin{equation}\label{240523-1559}
       \sigma_s(W[u]^T)(x_j)\leq \sigma_s(W_1[u])(x_j),~\forall s=1,\dots,l-1.
   \end{equation}
   Here we use the concavity and homogeneity of $\sigma_s^{1/s}$.
   
   By a computation and using \eqref{derivative-test}, we have 
   \begin{equation*}
   \nabla_T^2 u(x_j) -\nabla_T^2 v(0)\geq \nabla_T^2 w_j(x_j)-\nabla_T^2 v(0)= \frac{\alpha}{|x'_j|^2}( I_{n-1}-2\frac{x'_j}{|x'_j|}\otimes \frac{x'_j}{|x'_j|})\in \overline{\Gamma_s},
   \end{equation*}
   for $s=1,\dots,l-1$.
   Note that $-\nabla^2_T v(0)-\frac{1}{2}(\frac{\p u}{\p x_n}(x_j))^2 I_{n-1}= W_1[u](x_j)+ (\nabla_T^2u(x_j)-\nabla_T^2v(0))$, we have
\begin{equation*}
     \sigma_s(W_1[u])(x_j)\leq 
     \sigma_s\big(-\nabla^2_T v(0)-\frac{1}{2}(\frac{\p u}{\p x_n}(x_j))^2 I_{n-1}\big)     ,~\forall s=1,\dots,l-1.
   \end{equation*}

   Combining this with \eqref{240523-1559}, we obtain that 
   \begin{equation}\label{240523-1749}
       \sigma_s(W[u]^T)(x_j)\leq \sigma_s\big(-\nabla^2_T v(0)-\frac{1}{2}(\frac{\p u}{\p x_n}(x_j))^2 I_{n-1}\big),~\forall s=1,\dots,l-1.
   \end{equation}
   
   Define
   \begin{equation*}
       P(t)\coloneqq -e^{-(2l-1)u(0)} t \cdot \sum_{s=0}^{l-1} \alpha(l,s)\sigma_s\big(-\nabla^2_T v(0)-\frac{1}{2} t^2 I_{n-1}\big) t^{2l-2s-2}.
    \end{equation*}
    
    By \eqref{240523-1749}, we have 
    \begin{align}
        \cB^l[u](x_j)&= -e^{-(2l-1)u(x_j)} \frac{\p u}{\p x_n}(x_j)\cdot \sum_{s=0}^{l-1} \alpha(l,s)\sigma_s(W[u]^T)(x_j)(\frac{\p u}{\p x_n}(x_j))^{2l-2s-2} \notag \\
        &\leq e^{(2l-1)(u(0)-u(x_j))} P(\frac{\p u}{\p x_n}(x_j))
        \leq (1+o_j(1)) P(\frac{\p u}{\p x_n}(x_j)), \label{240523-1758}
    \end{align}
        where $o_j(1)$ tends to $0$ as $j\to \infty$. In the last inequality, we used $u(0)=\liminf_{x\to 0}u(x)$. Since $v$ is $C^2$,
    \begin{equation*}
         \cB^l[v](x_j)=
        P(\frac{\p v}{\p x_n}(0))+ o_j(1).
    \end{equation*}
 By the above, \eqref{240523-1758}, and \eqref{equ-mp-singular}, we have 
    \begin{equation}\label{240523-2125}
        P(\frac{\p v}{\p x_n}(0)) \leq P(\frac{\p u}{\p x_n}(x_j)) + o_j(1).
    \end{equation}

    Next, we prove the following formula:
    \begin{equation}\label{derivative-formula}
        P'(t)\equiv -e^{-(2l-1)u(0)} \sigma_{l-1}\big( -\nabla_T^2 v(0)-\frac{1}{2}t^2 I_{n-1}\big).
    \end{equation}
    Indeed, by the chain rule,
    \begin{equation*}
        P'(t)= -e^{-(2l-1)u(0)}( I +II),
    \end{equation*}
    where 
    \begin{equation*}
        I\coloneqq -\sum_{s=1}^{l-1} \bigg( \alpha(l,s)\cdot \sum_{\alpha=1}^{n-1}\frac{\p \sigma_s}{\p M_{\alpha\alpha}}\big(-\nabla^2_T v(0)-\frac{1}{2} t^2 I_{n-1}\big)\cdot t^{2l-2s}  \bigg),
    \end{equation*}
    and
    \begin{equation*}
        II\coloneqq \sum_{s=0}^{l-1} \alpha(l,s)(2l-2s-1)\sigma_s\big(-\nabla^2_T v(0)-\frac{1}{2} t^2 I_{n-1}\big) t^{2l-2s-2}.    \end{equation*}
    Recall that for $M\in \cS^{(n-1)\times (n-1)}$, $\sum_{\alpha=1}^{n-1} \frac{\p \sigma_s}{\p M_{\alpha\alpha}}(M)=(n-s) \sigma_{s-1}(M)$. 
    Hence,
    \begin{align*}
    I
    &=
    -\sum_{s=1}^{l-1}  \alpha(l,s) (n-s) \sigma_{s-1} \big(-\nabla^2_T v(0)-\frac{1}{2} t^2 I_{n-1}\big) t^{2l-2s}
    \\&=
    -\sum_{s=0}^{l-2}  \alpha(l,s+1) (n-s-1) \sigma_{s} \big(-\nabla^2_T v(0)-\frac{1}{2} t^2 I_{n-1}\big) t^{2l-2s-2}. 
    \end{align*}
    
   Since $\alpha(l,s+1)(n-s-1)=\alpha(l,s)(2l-2s-1)$ for $s=0,\dots,l-2$, combining the above computations gives the formula \eqref{derivative-formula}. 

   Since $W[v](0)= e^{2 v(0)}A[v](0)\in\Gamma_k$ and $\nabla_T v(0) = 0$, we have 
   $(W[v]^T(0)=) -\nabla^2_T v(0) - 2^{-1}(\frac{\p v}{\p x_n} (0) )^2 I_{n-1} \in\Gamma_{k-1}$. Hence, $-\nabla^2_T v(0) - 2^{-1} t^2 I_{n-1} \in \Gamma_{k-1}$ for all $t\in [\frac{\p v}{\p x_n} (0), 0]$. From \eqref{derivative-formula} and $l \leq k$, we obtain
   \begin{equation}\label{polynomial-mono}
       P'(t)<0,~ \forall t\in \left[\frac{\p v}{\p x_n}(0), 0\right].
   \end{equation}

   Since $0\geq \frac{\p u}{\p x_n}(x_j)\geq \frac{\p w_j}{\p x_n}(x_j)=\frac{\p v}{\p x_n}(0)+\delta$, we reach a contradiction in view of \eqref{240523-2125} and \eqref{polynomial-mono}.

Therefore, Step 2 is completed.

\smallskip

Sending $\alpha$ to $0$ in \eqref{order-claim}, we obtain that for any $x\in \overline{B^+_r}\setminus\{0\}$, and any $0<r<r_0$, 
\begin{equation*}
    u(x)\geq v_0(x)+ \inf_{x\in \p'' B_r^+}\left(  
        u(x)-v_0(x) \right).
\end{equation*}
Sending $x$ to $0$, we obtain the desired conclusion of Lemma \ref{key-lem}.
\end{proof} 

Now we return to the proof of Proposition \ref{mp-singular} under the assumption $\nabla_T v(0)=0$. 

For $\epsi>0$ small, let
\begin{equation*}
    v_\epsi (x)\coloneqq v(x)+\epsi (|x|+x_n).
\end{equation*}

We prove that 
\begin{equation}\label{app-key-lem}
    C_\epsi\coloneqq - \inf_{x\in \overline{B^+}} (u-v_\epsi)(x)>0.
\end{equation}
Indeed, by Lemma \ref{key-lem} (with $\delta=\epsi$ there), there exists a sequence of points $\{x^j\}\subset \overline{B^+}\setminus\{0\}$ with $x^j\to 0$ as $j\to\infty$ such that
\begin{equation*}
    u(x^j)-u(0)-\left( \frac{\p v}{\p x_n}(0)+\epsi\right)x_n^j +o(|x^j|)\leq 0\quad \text{for all}~j.
\end{equation*}
From the above, we have for large $j$,
\begin{equation*}
    u(x^j)-v_\epsi(x^j)=u(x^j)-u(0)-\left( \frac{\p v}{\p x_n}(0)+\epsi\right)x_n^j -\epsi |x^j|+o(|x^j|)\leq -\frac{\epsi}{2}|x^j|<0.
\end{equation*}
This implies \eqref{app-key-lem}. 

By \eqref{app-key-lem}, there exists some $x_\epsi\in \overline{B^+}\setminus\{0\}$ such that $-(u-v_\epsi)(x_\epsi)=C_\epsi$. Let
\begin{equation*}
    \widetilde{v_\epsi}(x)\coloneqq v_\epsi(x)-C_\epsi= v(x)+\epsi (|x|+x_n)-C_\epsi,
\end{equation*}
then $\widetilde{v_\epsi}$ touches $u$ from below at $x_\epsi$, i.e. $\widetilde{v_\epsi}\leq u$ in $\overline{B^+}$ and $\widetilde{v_\epsi}= u$ at $x_\epsi$. 
Using the positivity of $u-v$ in $\overline{B^+}\setminus\{0\}$, we deduce from the above that
\begin{equation}\label{240130-1357-av}
    C_\epsi= -(u-v_\epsi)(x_\epsi)= v(x_\epsi)-u(x_\epsi)+\epsi (|x_\epsi|+(x_\epsi)_n)\leq 2\epsi |x_\epsi|,
\end{equation}
and 
\begin{equation}\label{240130-1348-av}
    \lim\limits_{\epsi\to 0}|x_\epsi|=0.
\end{equation}

There are two possibilities: either $x_\epsi\in B^+$ or $x_\epsi\in \p'B^+\setminus\{0\}$. We will derive a contradiction under both two cases when $\epsi$ is small.

\textbf{Case 1:} $x_\epsi\in B^+$. 

We will show, for small $\epsi>0$, the following holds:
\begin{equation}\label{interior-contra}
\begin{cases}
    \text{either}~``  
    e^{p v(x_\epsi)}   \cdot \sigma_k(A[v])(x_\epsi)
    >
e^{p\widetilde{v_\epsi}(x_\epsi)}\cdot \sigma_k(A[\widetilde{v_\epsi}])(x_\epsi),  \\            A[\widetilde{v_\epsi}](x_\epsi)\in\Gamma_k"~\text{or}~``
    A[\widetilde{v_\epsi}](x_\epsi)\notin\Gamma_k".
    \end{cases}
\end{equation}
This would lead to a contradiction. Indeed,
by the derivative tests, we have
$\nabla \widetilde{v_\epsi}(x_\epsi)=\nabla u(x_\epsi)$ and $\nabla^2 \widetilde{v_\epsi}(x_\epsi)\leq\nabla^2 u(x_\epsi)$. 
Using the equation \eqref{equ-mp-singular}, we obtain $e^{p\widetilde{v_\epsi}(x_\epsi)}\cdot  \sigma_k(A[\widetilde{v_\epsi}])(x_\epsi) 
\geq e^{p v(x_\epsi)}\cdot \sigma_k(A[v])(x_\epsi) 
$ and $A[\widetilde{v_\epsi}](x_\epsi)\in\Gamma_k$. A contradiction.

Assertion \eqref{interior-contra} can be proved similarly to \cite[Proof of Theorem 1.1]{CLN1}. For any $\epsi>0$, we only need to treat the case when $\lambda(A[\widetilde{v_\epsi}])(x_\epsi)\in\Gamma_k$. By a computation and \eqref{240130-1357-av}, we have 
\begin{equation}\label{240212-1253-av}
    |\widetilde{v_\epsi}- v|(x_\epsi)+|\nabla (\widetilde{v_\epsi}-v)|(x_\epsi)=O(\epsi),
\end{equation}
and, for some $O(x)\in O(n)$, 
\begin{equation}\label{240212-1527}
    \nabla^2(\widetilde{v_\epsi}- v)(x)=\frac{\epsi}{|x|} O(x)^t \cdot \operatorname{diag}\{1,\dots,1,0\}\cdot  O(x).
\end{equation}

Denote 
\begin{equation*}
F(s,p,M)\coloneqq e^{ps}\sigma_k(e^{-2s}(-M+p\otimes p-\frac{1}{2}|p|^2 I)),~(s,p,M)\in \bR\times\bR^n\times\cS^{n\times n}.
\end{equation*}
There exists $b,\delta>0$ such that the set
\begin{equation*}
    S\coloneqq \{(s,p,M)\mid \exists x\in B^+_{1/2},~\|(s,p,M)-(v,\nabla v,\nabla^2 v)(x)\|\leq \delta \}
\end{equation*}
satisfies 
\begin{equation} \label{eqn-241007-0400}
   -F_{M_{ij}}(s,p,M)\geq b\delta_{ij},~ \forall (s,p,M)\in S. 
\end{equation}

\noindent Define
\begin{equation*}
    \overline{t_\epsi}\coloneqq
    \begin{cases}
        0\quad \text{if}~ \|\nabla^2 (\widetilde{v_\epsi}-v)(x_\epsi)\|\leq \delta/5, \\
        1- \delta/(5\|\nabla^2 (\widetilde{v_\epsi}- v)(x_\epsi)\|)\quad \text{if}~ \|\nabla^2 ( \widetilde{v_\epsi}- v)(x_\epsi)\|> \delta/5.
    \end{cases}
\end{equation*}
Then by \eqref{240212-1253-av} and the definition of $\overline{t_\epsi}$,
\begin{equation} \label{eqn-241007-0416}
    ( \widetilde{v_\epsi},\nabla \widetilde{v_\epsi}, \nabla^2 v+(1-t)(\nabla^2\widetilde{v_\epsi}-v))(x_\epsi)\in S,~\forall \overline{t_\epsi}\leq t\leq 1.
\end{equation}

Now, by \eqref{240212-1253-av}, there exists some $C>0$ independent of $\epsi$,

\begin{align}
    e^{p v(x_\epsi)} \sigma_k(A[v])(x_\epsi)
    &-
    e^{p\widetilde{v_\epsi}(x_\epsi)} \sigma_k(A[\widetilde{v_\epsi}])(x_\epsi)
    \notag
    \\&\quad =
    F( v, \nabla v, \nabla^2 v)(x_\epsi)
    -
    F(\widetilde{v_\epsi}, \nabla \widetilde{v_\epsi}, \nabla^2 \widetilde{v_\epsi})(x_\epsi)
    \notag
    \\&\quad \geq
    F( \widetilde{v_\epsi}, \nabla \widetilde{v_\epsi}, \nabla^2 v)(x_\epsi) - F( \widetilde{v_\epsi}, \nabla \widetilde{v_\epsi} , \nabla^2 \widetilde{v_\epsi}  ) (x_\epsi) - C\epsi.
    \label{240212-1630}  
    \end{align}
Note that $\nabla^2 {v}   + (1-t)\nabla^2( \widetilde{v_\epsi}-v)=\nabla^2 \widetilde{v_\epsi}+t(\nabla^2 v-\nabla^2 \widetilde{v_\epsi})$. Using \eqref{240212-1527}, $\lambda(A[\widetilde{v_\epsi}])(x_\epsi)\in \Gamma_k$, and $\Gamma_k+ \overline{\Gamma_n} \subset \Gamma_k$, we obtain
\begin{align*}
 F( \widetilde{v_\epsi}, & \nabla \widetilde{v_\epsi}, \nabla^2 v)(x_\epsi)    - F( \widetilde{v_\epsi}, \nabla \widetilde{v_\epsi}, \nabla^2 \widetilde{v_\epsi})(x_\epsi) 
 \\
 & = 
 \int_0^1 (-F_{M_{ij}})( \widetilde{v_\epsi}, \nabla \widetilde{v_\epsi}, \nabla^2 {v}   + (1-t)\nabla^2( \widetilde{v_\epsi}-v))(x_\epsi) dt \cdot  (\widetilde{v_\epsi}-v)_{ij}(x_\epsi)
\\&\geq
\int_{\overline{t_\epsi}}^1 (-F_{M_{ij}})( \widetilde{v_\epsi}, \nabla \widetilde{v_\epsi}, \nabla^2 {v}   + (1-t)\nabla^2( \widetilde{v_\epsi}-v))(x_\epsi) dt \cdot  (\widetilde{v_\epsi}-v)_{ij}(x_\epsi). 
\end{align*}
Combining this with \eqref{eqn-241007-0400} and \eqref{eqn-241007-0416}, we obtain
\begin{align*}
    &F( \widetilde{v_\epsi}, \nabla \widetilde{v_\epsi}, \nabla^2 v)(x_\epsi)    - F( \widetilde{v_\epsi}, \nabla \widetilde{v_\epsi}, \nabla^2 \widetilde{v_\epsi})(x_\epsi) 
    \\&\quad \geq 
    b (1-\overline{t_\epsi}) \cdot  \Delta (\widetilde{v_\epsi}-v) (x_\epsi)
    \\&\quad \geq 
    b(1-\overline{t_\epsi}) \|\nabla^2   (\widetilde{v_\epsi}-v) (x_\epsi) \|
    \geq 
    b \min\{\|\nabla^2   (\widetilde{v_\epsi}-v) (x_\epsi) \|, \delta/5\}\geq b \min\{\epsi/|x_\epsi|,\delta/5\}.
\end{align*}
Here in the last line, we also used \eqref{240212-1527} and the definition of $\overline{t_\epsi}$.
Combining \eqref{240130-1348-av}, \eqref{240212-1630}, and the above, assertion \eqref{interior-contra} is proved for small enough $\epsi$.

\smallskip

\textbf{Case 2:} $x_\epsi\in\p' B^+\setminus\{0\}$. 

By \eqref{240130-1357-av}, \eqref{240130-1348-av} and the derivative test, we have, at $x_\epsi$, 
\begin{equation}\label{derivative-test-2}
    \begin{cases}
    u=\widetilde{v_\epsi}=v(x_\epsi)+O(\epsi|x_\epsi|),\\
    
    \nabla_T u=\nabla_T \widetilde{v_\epsi}=\nabla_T v(x_\epsi)+\epsi \frac{x_\epsi^\prime}{|x_{\epsi}^{\prime} |},\quad
        \frac{\p u}{\p x_n}\geq\frac{\p \widetilde{v_\epsi}}{\p x_n}=\frac{\p v}{\p x_n}(x_\epsi)+\epsi,\\
        \nabla_T^2 u\geq \nabla_T^2 \widetilde{v_\epsi}=\nabla_T^2 v(x_\epsi) +\frac{\epsi}{|x_\epsi^\prime|}( I_{n-1}-\frac{x_\epsi^\prime}{|x_\epsi^\prime|}\otimes \frac{x_\epsi^\prime}{|x_\epsi^\prime|}).
        \end{cases}
\end{equation}

Recall that $W[u]^T=W_1[u]+W_2[u]$ as in \eqref{240919-2326}. By \eqref{derivative-test-2}, we have at $x_\epsi$, 
\begin{equation*}
    W_1[u]\leq 
-\nabla_T^2 v -\frac{\epsi}{|x_\epsi^\prime|}( I_{n-1}-\frac{x_\epsi^\prime}{|x_\epsi^\prime|}\otimes \frac{x_\epsi^\prime}{|x_\epsi^\prime|})-\frac{1}{2}(\frac{\p u}{\p x_n})^2 I_{n-1}, 
\end{equation*}
\begin{equation*}
W_2[u]\leq \nabla_T v\otimes \nabla_T v-\frac{1}{2}|\nabla_T v|^2 I_{n-1}+O(\epsi).    
\end{equation*}
Therefore, we obtain that at $x_\epsi$,
\begin{equation}\label{240524-1324}
    W[u]^T\leq W[v]^{(n-1)} -\frac{1}{2}(\frac{\p u}{\p x_n})^2 I_{n-1} +E,
\end{equation}
where 
\begin{equation*}
    W[v]^{(n-1)}\coloneqq -\nabla_T^2 v+\nabla_T v\otimes \nabla_T v -\frac{1}{2}|\nabla_T v|^2 I_{n-1},
\end{equation*}
and
\begin{equation*}
    E\coloneqq -\frac{\epsi}{|x_\epsi^\prime|} \left(I_{n-1}-\frac{x_\epsi^\prime}{|x_\epsi^\prime|}\otimes \frac{x_\epsi^\prime}{|x_\epsi^\prime|} +O(|x_\epsi^\prime|)\right).
\end{equation*}
The eigenvalues of $E$ takes the form of $\lambda(E)=-\frac{\epsi}{|x_\epsi^\prime|}\{(1,\dots,1,0)+O(|x_\epsi^\prime|)\}$. It is clear that $(1,\dots,1,0)$ is an interior point of $\Gamma_s$ in $\bR^{n-1}$ for any $s=1,\dots,n-2$. By $2\leq l\leq n/2$ and \eqref{240130-1348-av}, we have $-E\in \Gamma_s$ for $s=1,\dots,l-1$, as $\epsi>0$ small. Combining this with \eqref{240524-1324}, 
we obtain that  
\begin{equation}\label{240524-1402}
    \sigma_s(W[u]^T)(x_\epsi)\leq \sigma_s\big(W[v]^{(n-1)}(x_\epsi)-\frac{1}{2}(\frac{\p u}{\p x_n}(x_\epsi))^2 I_{n-1}\big),~\forall s=1,\dots,l-1.
\end{equation}

   Define 
   \begin{equation*}
       P_\epsi(t)\coloneqq -e^{-(2l-1)v(x_\epsi)} t \cdot \sum_{s=0}^{l-1} \alpha(l,s)\sigma_s\big( W[v]^{(n-1)}(x_\epsi)-\frac{1}{2} t^2 I_{n-1}\big) t^{2l-2s-2}.
    \end{equation*}

    By \eqref{240524-1402}, we have 
    \begin{gather}
        \cB^l[u](x_\epsi)= -e^{-(2l-1)u(x_\epsi)} \frac{\p u}{\p x_n}(x_\epsi)\cdot \sum_{s=0}^{l-1} \alpha(l,s)\sigma_s(W[u]^T)(x_\epsi)(\frac{\p u}{\p x_n}(x_\epsi))^{2l-2s-2} \notag \\
        \leq  e^{(2l-1)(v-u)(x_\epsi)} P_\epsi(\frac{\p u}{\p x_n}(x_\epsi)). \label{240524-1405}
    \end{gather}
By the boundary condition of \eqref{equ-mp-singular} and \eqref{derivative-test-2}, we have
\begin{equation}\label{240524-1418}
   \cB^l[u](x_\epsi)\geq \cB^l[v](x_\epsi)= P_\epsi(\frac{\p v}{\p x_n}(x_\epsi)).
\end{equation}

A similar proof as \eqref{derivative-formula} gives
\begin{equation}\label{derivative-formula-2}
        P_\epsi^\prime(t)\equiv -e^{-(2l-1)v(x_\epsi)} \sigma_{l-1}\big( W[v]^{(n-1)}(x_\epsi)-\frac{1}{2}t^2 I_{n-1}\big).
\end{equation}
Since $A[v](x_\epsi)\in\Gamma_k$, we have $W[v]^T(x_\epsi)\in\Gamma_{k-1}$. By \eqref{derivative-formula-2}, this implies $P_\epsi^\prime(\frac{\p v}{\p x_n}(x_\epsi))=-e^{-(2l-1)v(x_\epsi)} \sigma_{l-1}(W[v]^T(x_\epsi))<0$. Using an argument similar to the one for \eqref{polynomial-mono}, we obtain $P_\epsi^\prime(t)<0$, $\forall t\in [\frac{\p v}{\p x_n}(x_\epsi), 0]$. Actually, using $v\in C^2(\overline{B^+})$,  the same argument leads to 
\begin{equation}\label{polynomial-mono-quant}
       P_\epsi^\prime (t)\leq -\delta_0,~\forall t\in \left[\frac{\p v}{\p x_n}(x_\epsi),0\right],
   \end{equation}
where $\delta_0>0$ is some constant depending on $v$. 
Since $\frac{\p v}{\p x_n}(x_\epsi)+\epsi\leq \frac{\p u}{\p x_n}(x_\epsi)\leq 0$, the above implies that 
\begin{equation}\label{240524-1450}
    P_\epsi(\frac{\p u}{\p x_n}(x_\epsi))\leq P_\epsi(\frac{\p v}{\p x_n}(x_\epsi)+\epsi).
\end{equation}

By \eqref{240524-1405}, \eqref{240524-1418}, and \eqref{240524-1450}, we obtain  
\begin{equation*}
 e^{(2l-1)(u-v)(x_\epsi)}  P_\epsi(\frac{\p v}{\p x_n}(x_\epsi))\leq 
P_\epsi(\frac{\p v}{\p x_n}(x_\epsi) +\epsi).
\end{equation*}
By \eqref{derivative-test-2}, the above implies that
\begin{equation*}
     P_\epsi(\frac{\p v}{\p x_n}(x_\epsi))+ O(\epsi|x_\epsi|)\leq 
     P_\epsi(\frac{\p v}{\p x_n}(x_\epsi))+ P_\epsi^\prime (\frac{\p v}{\p x_n}(x_\epsi))\epsi+O(\epsi^2).
\end{equation*}
Applying \eqref{polynomial-mono-quant}, we get
    $o(\epsi)\leq -\delta_0 \epsi$, which leads to a contradiction when $\epsi>0$ is small. 

\end{proof}

\section{Proof of Theorem \ref{bk-liouville}}\label{liouville-sec}

\begin{proof}[Proof of Theorem \ref{bk-liouville}]
Our proof uses the by-now-standard method of moving spheres, a variant of the method of moving planes. The crucial new ingredient is to handle the isolated boundary singularity as in \S \ref{mp-singular-sec}.
Let
\begin{equation*}
    v^{x,\lambda}(y)\coloneqq 2\log(\frac{\lambda}{|y-x|})+v(x+\frac{\lambda^2(y-x)}{|y-x|^2}).
\end{equation*}

\textbf{Step 1:} Prove that for any $x\in \p \hrn$, there exists $\lambda_0(x)>0$ such that 
\begin{equation*}
    v^{x,\lambda}\leq v ~\text{on}~\hrn\setminus B_\lambda (x),~\forall 0<\lambda<\lambda_0(x).
\end{equation*}

We first show that 
\begin{equation}\label{starter}
    \liminf_{|x|\to \infty} (v(x)+2\log |x|)>-\infty.
\end{equation}
Since $\av\in \Gamma_1$ in $\hrn$ and $\frac{\p v}{\p x_n}\leq 0$ on $\p\hrn$, by \cite[Lemma 4]{Li-Li_JEMS} (apply to $u=e^{\frac{n-2}{2}v}$ there), we obtain \eqref{starter}.

The assertion of Step 1 can be proved by a similar argument as in \cite[Proof of Lemma 2.1]{MR2001065}. For reader's convenience, we include a proof below.
Without loss of generality, take $x=0$, and denote $v^{0,\lambda}$ by $v^\lambda$. By the regularity of $v$, there exists $r_0>0$ such that 
\begin{equation*}
    \frac{d}{dr}(v(r,\theta)+\log r)>0,~\forall 0<r<r_0,~\theta\in S^{n-1}\cap \hrn.
\end{equation*}
It follows from the above that 
\begin{equation}\label{240212-1738}
    v^{\lambda}(y)< v(y),~\forall 0<\lambda<|y|<r_0,~y\in \hrn.
\end{equation}
Because of \eqref{starter}, there exists some constant $\alpha\in \bR$ such that 
\begin{equation*}
    v(y)\geq \alpha-2\log|y|,~\forall |y|\geq r_0,~y\in\hrn.
\end{equation*}
Let $\lambda_0\coloneqq \min\{\exp(\frac{1}{2}(\alpha-\max\limits_{\overline{B^+_{r_0}}}v)),r_0\}$. Then 
\begin{equation*}
    v^\lambda(y)\leq 2\log(\frac{\lambda_0}{|y|})+\max_{\overline{B^{+}_{r_0}}} v \leq \alpha -2\log|y|\leq v(y),~\forall 0<\lambda<\lambda_0,~y\in\hrn\setminus B_{r_0}.
\end{equation*}
Combining the above and \eqref{240212-1738}, Step $1$ is completed.

\medskip

By Step 1, 
\begin{equation*}
    \bar\lambda(x)\coloneqq \sup\{\mu>0\mid v^{x,\lambda}\leq v ~\text{on}~\overline{\hrn}\setminus B_\lambda(x),~\forall 0<\lambda<\mu\}\in (0,\infty],~\forall x\in\p\hrn.
\end{equation*}

\textbf{Step 2:} Prove that if $\bar\lambda(x)<\infty$ for some $x\in\p\hrn$, then $v^{x,\bar\lambda(x)}\equiv v$ in $\overline \hrn\setminus\{x\}$.

    Without loss of generality, we take $x=0$, and let $\bar{\lambda}=\bar{\lambda}(0)$ and $v^{\bar{\lambda}}=v^{0,\bar{\lambda}}$. By the definition of $\bar{\lambda}$ and the continuity of $v$,
    \begin{equation*}
        v^{\bar{\lambda}}\leq v ~~\text{in}~\overline{\hrn}\setminus B_{\bar{\lambda}}.
    \end{equation*}  
A calculation gives, using \eqref{equ-mp-singular} and the conformal invariance,
\begin{equation}\label{eqn-230528-1202}
\sigma_k(A[v^{\bar{\lambda}}])= (\frac{\bar\lambda}{|y|})^{2p} e^{-p v^{\bar\lambda}}\leq e^{-p v^{\bar\lambda}}
        ~\text{in}\,\, \overline\hrn \setminus {B_{\bar{\lambda}}},
        ~\text{and}~
        \cB^l[v^{\bar\lambda}]=c~\text{on}~\p\hrn\setminus B_{\bar{\lambda}}.
    \end{equation}
Note also from \eqref{equ-mp-singular}
   \begin{equation*}
\sigma_k(A[v])= e^{-p v}
        ~\text{in}\,\, \overline\hrn \setminus {B_{\bar{\lambda}}},
        ~\text{and}~
        \cB^l[v]=c~\text{on}~\p\hrn\setminus B_{\bar{\lambda}}.
    \end{equation*}
From the above and \eqref{eqn-230528-1202}, we have
\begin{equation*}
    \sigma_k(A[v^{\bar{\lambda}}])-\sigma_k(\av)
    -(e^{-p v^{\bar\lambda}}-e^{-pv})\leq 0\quad \text{in}~\overline\hrn\setminus B_{\bar\lambda}.   
\end{equation*}
Hence, by the strong maximum principle and Proposition \ref{mp-regular} (with $\psi(\cdot,v,\nabla v)=-e^{-pv}$ there), we may assume 
\begin{equation} \label{eqn-230528-1207}
    v > v^{\bar{\lambda}} \,\,\text{in}\,\,\overline\hrn \setminus \overline{B_{\bar{\lambda}}}
\end{equation}
since otherwise $v^{\bar{\lambda}} \equiv v$ on $\overline\hrn \setminus {B_{\bar{\lambda}}}$, which gives the desired conclusion of Step 2.

An application of the Hopf Lemma as that in the proof of \cite[Lemma 3]{Li-Li_JEMS} gives 
\begin{equation} \label{eqn-230528-1212}
    \frac{\p}{\p \nu} (v - v^{\bar{\lambda}}) > 0
    \quad
    \text{on}\,\,
    \p B_{\bar{\lambda}}\cap \hrn,
\end{equation}
where $\nu$ is the unit outer normal to $\p B_{\bar\lambda}$.

Next we prove that
\begin{equation} \label{eqn-230528-1212-corner}
    \frac{\p}{\p \nu} (v - v^{\bar{\lambda}}) > 0
    \quad
    \text{on}\,\,
    \p B_{\bar{\lambda}}\cap \p\hrn.
\end{equation}
By the rotation symmetry, we only need to prove the above at point $(\bar{\lambda},0,\dots,0)$. The above follows from Proposition \ref{lem-corner-hopf}.

Finally, we prove that
\begin{equation} \label{eqn-230528-1209}         \liminf_{|y|\rightarrow \infty} (v-v^{\bar{\lambda}})(y) > 0.
\end{equation}
Similar to \eqref{eqn-230528-1202}, we have
     \begin{equation*}
\sigma_k(A[v^{\bar{\lambda}}])\geq e^{-p v^{\bar\lambda}}
        ~\text{in}\,\, B^+_{\bar{\lambda}},
        ~\text{and}~
        \cB^l[v^{\bar\lambda}]=c~\text{on}~\p' B^{+}_{\bar{\lambda}}\setminus \{0\}.
    \end{equation*}
By Proposition \ref{mp-singular}, we have 
    \begin{equation*}
        \liminf_{x\to 0}(v^{\bar{\lambda}}-v)(x)>0,
    \end{equation*}
i.e. \eqref{eqn-230528-1209} holds.
    It follows from \eqref{eqn-230528-1207}, \eqref{eqn-230528-1212}, \eqref{eqn-230528-1212-corner}, and \eqref{eqn-230528-1209} (see the proof of \cite[Lemma 2.2]{MR2001065} and 
    \cite[Lemma 3.2]{Li2021ALT}) that for some $\epsi>0$, $v^{0,\lambda}(y)\leq v(y)$, for every $y\in\overline\hrn\setminus B_{\lambda}^+$ and $\bar{\lambda} \leq \lambda < \bar{\lambda} + \epsi$, violating the definition of $\bar{\lambda}$.

\textbf{Step 3:} Prove that either $\bar{\lambda}(x)<\infty$ for all $x\in \p\hrn$ or $\bar{\lambda}(x)=\infty$ for all $x\in\p\hrn$.

Suppose for some $x\in \p\hrn$, $\bar{\lambda}(x)=\infty$. By definition, $v^{x,\lambda}(y) \leq v(y)$ on $\overline\hrn \setminus B_{\lambda}(x)$ for all $\lambda>0$. Adding both sides by $2\log |y|$ and sending $|y|$ to $\infty$, we have $v(x)+2\log \lambda \leq \liminf_{|y|\rightarrow \infty} (v(y)+2 \log|y|) $ for all $\lambda>0$. Hence, $\lim_{|y|\rightarrow \infty} (v(y)+2\log|y|) = \infty$.

On the other hand, if $\bar{\lambda}(x) < \infty$ for some $x\in\p\hrn$, we have, by Step 2, $v^{x,\bar{\lambda}(x)}\equiv v$ in $\hrn\setminus\{x\}$. Therefore, $\lim_{|y|\rightarrow \infty} (v(y)+2\log|y|) = v(x)+2\log \bar{\lambda}(x) < \infty$. Step 3 is proved.

\medskip

By Step 3, the remainder of the proof is divided into two cases. 

\textbf{Case 1:} $\bar{\lambda}(x)=\infty$ for all $x\in\p\hrn$.

Consequently, $v^{x,\lambda} (y) \leq v(y)$ for all $x\in\p\hrn$, $y\in\hrn\setminus B_{\lambda}(x)$, and $\lambda>0$. 
 This implies, by \cite[lemma 11.3]{MR2001065}, that 
$v(x',x_n)\equiv v(0',x_n)$, $\forall x'\in\bR^{n-1}$, $x_n\geq 0$. A direct computation shows that the eigenvalues of $A[v]$ take the following form
\begin{equation*}
    \lambda(\av)=(\lambda_1,\lambda_2,\dots,\lambda_2), 
\end{equation*}
where $\lambda_1=2^{-1}(v')^2 e^{-2 v} (-2(v')^{-2}v''+1)$ and $\lambda_2=-2^{-1}(v')^2 e^{-2v}$. Since $A[v]\in \Gamma_k$ in $\overline\hrn$, we have $\lambda_2 I_{n-1}\in \Gamma_{k-1}\subset\Gamma_1$. However, $\sigma_1(\lambda_2 I_{n-1})=-(n-1)2^{-1}(v')^2 e^{-2v}\leq 0$. We reach a contradiction and thus Case 1 does not occur.

\textbf{Case 2:} $\bar{\lambda}(x)<\infty$ for all $x\in \p\hrn$. 

By Step 2, 
\begin{equation}\label{240216-1706}
    v^{x,\bar{\lambda}(x)} \equiv v~\text{in}~\overline\hrn\setminus\{x\},~\forall x\in\p\hrn.
\end{equation}
It follows, by \cite[Lemma 2.5]{Li-Zhu} (see also \cite[Lemma A.2]{Li2021ALT}), that
\begin{equation}\label{240216-1707}
    v(x',0)=\log \left(\frac{\hat{a}}{|x'-\bar{x}'|^2+d^2} \right),~x\in\p\hrn, 
\end{equation}
for some $\bar{x}\in\p\hrn$ and $\hat{a},d>0$.

It also follows from \eqref{240216-1706}, by equation \eqref{equ-liouville-bk} and the M\"obius invariance, that we must have $p=0$.

Consider spheres $\p B_{\bar{\lambda}(x)}(x)$ for $x\in\p\hrn$. Using \eqref{240216-1706} at $\infty$ and the explicit formula \eqref{240216-1707}, we obtain that $P,Q\in \p B_{\bar{\lambda}(x)}(x)$, $x\in\p\hrn$, where $P=(\bar{x}',-d)$ and $Q=(\bar{x}',d)$. Consider the M\"obius transform 
\begin{equation*}
\psi(z)\coloneqq P+\frac{4 d^2 (z-P)}{|z-P|^2}~ \text{and} 
~w(z)\coloneqq v^{P,2d}(z)=2\log\left( \frac{2d}{|z-P|} \right) + v\circ\psi (z).
\end{equation*}
By the properties of M\"obius transforms, $\psi(\hrn)=B_{2d}(Q)$ and $\psi$ maps every sphere $\p B_{\bar{\lambda}(x)}(x)$, $x\in\p\hrn$, to every hyperplane through $Q$. By \eqref{240216-1706}, $w$ is symmetric with respect to all hyperplanes through $Q$, and thus $w$ is radially symmetric about $Q$ in $B_{2d}(Q)$.

By equation \eqref{equ-liouville-bk} and its M\"obius invariance, 
\begin{equation*}
    \sigma_k(A[w])=1\quad \text{in}~B_{2d}(Q).
\end{equation*}
By \cite[Theorem 3]{Li-Li_JEMS},
\begin{equation*}
    w(z)\equiv \log \big( \frac{\bar{a}}{1+\bar{b}|z-Q|^2} \big)\quad \text{in}~B_{2d}(Q).
\end{equation*}
Since $v(x)=w^{P,2d}(x)$, together with \eqref{240216-1707}, we have
\begin{equation*}
    v(x',x_n)\equiv \log \big(\frac{a}{1+b|(x',x_n)-(\bar x',\bar x_n )|^2}\big),
\end{equation*}
where $a=(1-s^2)^{-1}d^{-2} \hat{a}$, $b=(1-s^2)^{-1}d^{-2}$, $\bar{x}_n=ds$, and $s=(4\bar{b}d^2-1)/(4\bar{b}d^2+1)$. Here $\hat{a}$, $d$ and $\bar{b}$ are positive numbers given above. It is easy to verify that $a$, $b$ and $\bar{x}_n$ have full range of $\bR_+$, $\bR_+$ and $\bR$, respectively. Plugging the above form of $v$ back to the equation, it is easy to verify that $v$ must take the form of \eqref{half-space-bubble}. 
\end{proof}

\noindent
{\it On Remark \ref{remark-1.1}:} If $\sigma_k$ in Theorem \ref{bk-liouville} is replaced by general $f$ as stated in Remark \ref{remark-1.1}, we should modify the above proof as follows. Step 1 remains unchanged.
In Step 2, by the same proofs, Proposition \ref{mp-regular}, \ref{lem-corner-hopf}, and \ref{mp-singular} remain valid if $\sigma_k$ is replaced by such $f$, and thus \eqref{eqn-230528-1207}, \eqref{eqn-230528-1212-corner}, and \eqref{eqn-230528-1209} follow respectively as before. After the unchanged Step 3, Case 1 can be ruled out as before since only $\av\in\Gamma_k$, $k\geq 2$, is used. Case 2 follows as before except that one should use \cite[Lemma 2.3]{CLL-2} instead of \cite[Theorem 3]{Li-Li_JEMS}.

\section{Local gradient estimates}\label{lgest-sec}

The proof of Theorem \ref{lgest-bk-eucildean} needs the following quantity introduced in \cite{Li2009}. For $w \in C^{0,1}_{loc}(\Omega)$ defined on an open subset $\Omega$ of $\overline\hrn$, $\gamma \in (0,1)$, and $x\in \Omega$, define
\begin{equation*}
    \delta(w,x;\Omega, \gamma)\coloneqq 
\begin{cases}
	\infty &\text{if}~~\dist(x,\p''\Omega)^\gamma [w]_{\gamma, \dist(x,\partial''\Omega)}(x) <1,\\
	\mu &\text{where}~~0<\mu \leq \dist(x,\partial''\Omega)~\text{and} ~\mu^\gamma [w]_{\gamma,\mu}(x)=1,
        \\&
        \text{if} ~~ 
	\dist(x,\p''\Omega)^\gamma  [w]_{\gamma, \dist(x,\partial''\Omega)} (x) \geq 1,
 \end{cases}
\end{equation*}
where $[w]_{\gamma,s}(x) \coloneqq \sup_{y\in B_s(x)\setminus\{x\}} |x-y|^{-\gamma}|w(y)-w(x)|$. In this notation, we follow the convention that $\sup$ is taken over the set $B_s(x)\setminus\{x\}$ intersecting where $w$ has its definition.  

\begin{proof}[Proof of Theorem \ref{lgest-bk-eucildean}]
We first prove that
\begin{equation}\label{boundary-normal-C1}
    -C e^{\sup_{B^+}v}\leq \frac{\p v}{\p x_n}\leq 0\quad \text{on}~\p' B^+,
\end{equation}
where $C$ is a constant depending only on $l,n$, and an upper bound of $c$. 

Indeed, since $\cB^l[v]\geq 0$ on $\p' B^+$, we have $h_{\bar{g}_{v}}\geq 0$, and hence $\frac{\p v}{\p x_n}\leq 0$ on $\p' B^+$. On the other hand, by $\av\in\Gamma_k$ on $\p' B^+$,
\begin{equation*}
    c=\cB^l[v]=\sum_{s=0}^{l-1}\alpha(l,s)\sigma_s(\av^T)h_{\bar{g}_v}^{2l-2s-1}\geq \alpha(l,0)h_{\bar{g}_v}^{2l-1}\quad\text{on}~\p' B^+.
\end{equation*}
This implies $h_{\bar{g}_v}\leq C$ on $\p' B^+$. The estimate \eqref{boundary-normal-C1} follows. 

\textbf{Step 1:} We prove the estimate \eqref{lgest-estimate} with $C$ additionally depending on a lower bound of $\inf_{B^+} v$.

We first prove, by the method of moving spheres, that
\begin{equation}\label{boundary-tangential-C1}
    |\nabla_T v|\leq C\quad \text{on}~\p' B_{1/2}^+.
\end{equation}

Denote
    $\eta\coloneqq \min\{10^{-1}\exp({2^{-1}(\min_{\p'' B_{3/4}^+}v-\max_{\overline{B_{1/2}^+}}v)}),2^{-1}\}>0$, then 
    \begin{equation}\label{faraway-notouch}
        v^{x,\lambda}< v~\text{on}~ \p''B_{3/4}^+,~\forall B_\lambda(x)\subset B_{\eta},~x\in\p\hrn.
    \end{equation}

By a similar argument as the proof of Step 1 of Theorem \ref{bk-liouville}, we can show that for any $x\in\p' B_{\eta}$, there exists some $\lambda_0(x)>0$, such that
\begin{equation*}
    v^{x,\lambda}\leq v~\text{on}~B_{3/4}^+\setminus B_\lambda(x),~\forall 0<\lambda<\lambda_0(x).
\end{equation*}

For $x\in\p' B_{\eta}$, let 
\begin{equation*}
    \bar\lambda(x)\coloneqq \sup\{0<\mu<\eta-|x|\mid v^{x,\lambda}\leq v~\text{on}~B_{3/4}^+\setminus B_\lambda(x),~\forall 0<\lambda<\mu\}\in (0,\eta-|x|].
\end{equation*}

 In the following, we claim that $\bar\lambda(x)=\eta-|x|$ for any $x\in\p' B_\eta$. If not, there exists some $x_0\in\p' B_\eta$ such that $\bar{\lambda}\coloneqq \bar\lambda(x_0)<\eta-|x_0|$. Using the strong maximum principle and Proposition \ref{mp-regular} as the arguments in Step 2 of the proof of Theorem \ref{bk-liouville}, we must have
\begin{equation} \label{eqn-230528-1207-lgest}
    v^{x,{\bar\lambda}}<v \,\,\text{in}\,\,(B_{3/4}^+\cup\p' B_{3/4}^+)\setminus \overline{B_{\bar{\lambda}}(x)}
\end{equation}
since otherwise $v^{x,\bar{\lambda}} \equiv v$ on $\overline{B_{3/4}^+}\setminus B_{\bar\lambda}(x)$, which contradicts \eqref{faraway-notouch}.

As before, an application of the Hopf Lemma and Proposition \ref{lem-corner-hopf} gives 
\begin{equation} \label{eqn-230528-1212-lgest}
    \frac{\p}{\p \nu} (v - v^{\bar{\lambda}}) > 0
    \quad
    \text{on}\,\,
    \p'' B_{\bar{\lambda}}(x)^+,
\end{equation}
where $\nu$ is the unit outer normal to $\p B_{\bar\lambda}(x)$.

It follows from \eqref{faraway-notouch}--\eqref{eqn-230528-1212-lgest} (see the proof of \cite[Lemma 2.2]{MR2001065} and 
    \cite[Lemma 3.2]{Li2021ALT}) that for some $\epsi>0$, $v^{x,\lambda}(y)\leq v(y)$, for every $y\in {B_{3/4}^+}\setminus B_{\lambda}(x)$ and $\bar{\lambda} \leq \lambda < \bar{\lambda} + \epsi$, violating the definition of $\bar{\lambda}$. Hence, the claim is proved. 

    The claim implies that
    \begin{equation*}
        v^{x,\lambda}\leq v~\text{on}~\p' B_\eta^+\setminus \p' B_\lambda(x),~\forall B_\lambda(x)\subset B_\eta,~x\in\p\hrn.
    \end{equation*}
 By \cite[Lemma 2]{LN09arxiv} (apply in $\bR^{n-1}=\p\hrn$),   
 \begin{equation*}
     |\nabla_T v(x)|\leq \frac{C}{\eta-|x|},\quad x\in \p' B_\eta^+.
 \end{equation*}
 Hence, the estimate \eqref{boundary-tangential-C1} follows.

 Now we prove, by Bernstein-type arguments, that
 \begin{equation}\label{interior-pts-lgest}
     |\nabla v|\leq B_{1/2}^+.
 \end{equation}
Consider 
\[ \Psi \coloneqq \rho \cdot e^{\phi (v)} \cdot|\nabla v|^2,\]
where $\rho$ is a cut-off function with 
\begin{equation*}
		\left\{
	\begin{array}{lr}
		 \rho = 1 ~~\text{in}~B_{1/2},\quad \rho=0 ~~\text{in}~ B_1\setminus B_{2/3}, \quad \rho >0 ~~\text{in}~ B_{2/3},\quad \rho \geq 0, ~~ \text{in}~B_1,\\[1ex]
		|\partial_k \rho| \leq C\sqrt{\rho} ~~ \text{in}~B_1,~~ \text{for some absolute constant $C$},~~\forall k,
	\end{array}
	\right.	
\end{equation*}
and $\phi(s)\coloneqq \epsilon e^{2 s}$ is an auxiliary function. Here, small constant $\epsi$ is chosen, depending only on $a\coloneqq \inf_{B^+}v$ and $b\coloneqq \sup_{B^+}v$, such that for some $c_1 = c_1 (a,b)$,
\begin{equation*}
      2^{-1} \phi' \geq c_1 >0,\quad \phi'' - \phi' -\left(\phi'\right)^2 \geq 0 \quad \text{on}~~ [a,b].
\end{equation*}
To prove the desired gradient estimate, it suffices to prove
\begin{equation*}
	\Psi(x_0) = \max\limits_{\overline{B^+}} \Psi \leq C .
\end{equation*}
Clearly, $x_0 \in \overline{B_{2/3}^+}$. If $x_0\in\p\hrn$, the desired gradient bound follows directly from \eqref{boundary-normal-C1} and \eqref{boundary-tangential-C1}. Hence, we may assume that $x_0\in B^+$. The remaining proof is the same as that of the interior gradient estimate of equation $\sigma_k^{1/k}(\av)=1$ (see, for example, \cite[Appendix]{Li2009}). The estimate \eqref{interior-pts-lgest} is thus proved. Step 1 is completed.

\textbf{Step 2:} We prove estimate \eqref{lgest-estimate} under a one-sided bound assumption on $v$.  

    For $\gamma\in (0,1)$, we prove $[v]_{C^\gamma(B^+_{1/2})}\leq C(\gamma)$ by contradiction. Suppose not, there exists a sequence $\{v_i\}$ satisfying \eqref{equ-lgest-bk-euclidean} in $B_2^+\cup \p'B_2^+$ (for convenience, we consider equation in a ball of radius $2$), but $[v_i]_{C^\gamma}(B^+_{1/2})\to\infty$. This implies    \begin{equation}\label{contra-d}
    	\inf\limits_{x\in B_{1/2}^+} \delta \left(v_i, x\right) \longrightarrow 0,
    \end{equation}
    where for convenience, we denote $\delta\left(v_i, x\right)\coloneqq \delta\left(v_i, x ; B_2^+\cup\p' B_2^+, \gamma \right)$.     
    From $\eqref{contra-d}$, there exists $x_i\in B_1^+\cup \p' B_1^+$, such that 
  \begin{equation*}
    \frac{1-|x_i|}{\delta\left(v_i, x_i\right)}= \sup\limits_{x\in \overline{B_1^+}} \frac{1-|x|}{\delta\left(v_i, x\right)}\longrightarrow \infty.
  \end{equation*}
  Let
  \begin{equation*}
  	\sigma_i \coloneqq \frac{1-|x_i|}{2},~~\epsilon_i\coloneqq \delta\left(v_i, x_i\right).
  \end{equation*}
Then we have 
\begin{equation}\label{scale-d}
	\frac{\sigma_i}{\epsilon_i}\rightarrow \infty,~~ \epsilon_i \rightarrow 0, 
\end{equation}
and 
\begin{equation}\label{scale2-d}
	\epsilon_i \leq 2 \delta\left(v_i, x\right)~, \forall\, x\in B_{\sigma_i}(x_i)\cap \overline{B_1^+}.
\end{equation}
Let
\begin{equation*}
\widetilde{v}_i(y)\coloneqq v_i(x_i +\epsilon_i y) - v_i (x_i),\quad |y|<\frac{\sigma_i}{\epsilon_i},~y_n\geq -T_i\coloneqq -\frac{1}{\epsi_i}(x_i)_n.
\end{equation*}
Then by definition we immediately see that 
$\widetilde{v}_i(0)=0$ and $[\widetilde{v}_i]_{\gamma, 1}(0)= \epsilon_i^\gamma [v_i]_{\gamma, \epsilon_i}(x_i) =1$, and by the triangle inequality, 
\begin{equation*}
[\widetilde{v}_i]_{\gamma, 1}(x) \leq 2^{-\gamma} \epsilon_i^\gamma \left(\sup\limits_{|z-(x_i +\epsilon_i x)|<\epsilon_i} [v_i]_{\gamma, \epsilon_i /2}(z) + [v_i]_{\gamma, \epsilon_i /2}(x_i +\epsilon_i x) \right),~ \forall |x| < \frac{\sigma_i}{2\epsi_i},~x_n\geq -T_i.
\end{equation*}
Combining with $\eqref{scale-d}$, $\eqref{scale2-d}$, 
we obtain that for large $i$,
\begin{equation*}
	\begin{array}{llll}
	[\widetilde{v}_i]_{\gamma, 1}(x) &\leq& \sup\limits_{|z-(x_i +\epsilon_i x)|<\epsilon_i} \delta(v_i, z)^\gamma [v_i]_{\gamma, \delta(v_i, z)}(z) \\
	&+& \delta(v_i, x_i +\epsilon_i x)^\gamma [v_i]_{\gamma, \delta(v_i, x_i +\epsilon_i x)}(x_i +\epsilon_i x) =2,
~ \forall |x| < \frac{\sigma_i}{2\epsi_i},~x_n\geq -T_i.
	\end{array}
\end{equation*}
Therefore, using also $\widetilde{v}_i (0)=0$, we obtain
\begin{equation*}
    |\widetilde{v}_i|\leq C(K)~~\text{in}~ B_K (0)\cap\{y_n\geq -T_i\},~~\forall\, K > 1.
\end{equation*}      
     
A calculation shows that $\widetilde{v}_i$ satisfies 
\begin{equation}\label{equv_i}
\begin{cases}
\sigma_k^{1/k}(A[\widetilde{v}_i]) = \epsi_i^2 e^{2 v_i(x_i)}\quad &\text{in}~B_{\sigma_i/ \epsilon_i}(0)\cap \{ y_n>-T_i\}, \\
\cB^l[\widetilde{v}_i]=c \cdot \epsi_i^{2l-1} e^{(2l-1)v_i(x_i)}\quad &\text{on}~ B_{\sigma_i/ \epsilon_i}(0)\cap \{ y_n=-T_i\}.
\end{cases}
\end{equation}
 Thus, by Step 1,  
\begin{equation*}
	|\nabla \widetilde{v}_i | \leq C(K) ~~\text{in}~B_K (0)\cap\{y_n\geq -T_i\}, ~~\forall\, K>1.
\end{equation*}

Passing to a subsequence, we have either 
\begin{equation*}
    \lim_{i\to\infty}(-T_i)=-\infty
    \quad \text{or}\quad \lim_{i\to\infty}(-T_i)=-T\in (-\infty,0].
\end{equation*}

\textbf{Case 1:} $\lim_{i\to\infty}(-T_i)=-\infty$. There exists some $\widetilde{v}\in C^{0,1}_{loc}(\overline{\hrn})$ such that, after passing to another subsequence,
\begin{equation}\label{limit-1}
	\widetilde{v}_i\longrightarrow \widetilde{v} ~~\text{in} ~C^{\widetilde{\gamma}}_{loc}(\overline{\hrn}),~~\forall\, \widetilde{\gamma} \in (0,1).
\end{equation}
Since $[\widetilde{v}_i]_{\gamma, 1}(0)=1$,
we have, by \eqref{limit-1} with $\widetilde{\gamma}>\gamma$, that $[\widetilde{v}]_{\gamma, 1}(0)=1$.
In particular, $\widetilde{v}$ cannot be a constant.
On the other hand, by 
 \eqref{scale-d},
 \eqref{limit-1}, $\epsi_i^2 e^{2 v_i(x_i)}\to 0$, and the cone property of $\p \Gamma_k$, we can deduce that
\begin{equation*}
	A[\widetilde{v}] \in \partial \Gamma_k ~\text{in}~ \bR^n~~\text{in the viscosity sense}.
\end{equation*}
By \cite[Theorem 1.4]{Li2009} (see also \cite[Theorem 1.2]{CLL-1}), $\widetilde{v}\equiv \text{constant}$. A contradiction.

\textbf{Case 2:} $-T\in (-\infty,0]$.
There exists some $\widetilde{v}\in C^{0,1}_{loc}(\overline{\hrn})$ such that, after passing to another subsequence,
\begin{equation}\label{limit}
	\widetilde{v}_i(\cdot +(0', T_i)) \longrightarrow \widetilde{v} ~~\text{in} ~C^{\widetilde{\gamma}}_{loc}(\overline{\hrn}),~~\forall\, \widetilde{\gamma} \in (0,1).
\end{equation}
Since $[\widetilde{v}_i]_{\gamma, 1}(0)=1$,
we have, by \eqref{limit} with $\widetilde{\gamma}>\gamma$, that $[\widetilde{v}]_{\gamma, 1}(0',T)=1$.
In particular, $\widetilde{v}$ cannot be a constant.

Without loss of generality, we may assume that $T_i=T=0$ for each $i$. Let $\widetilde{v}_i(x',x_n)\coloneqq \widetilde{v}_i(x',-x_n)$, and $\widetilde{v}(x',x_n)\coloneqq \widetilde{v}(x',-x_n)$ for $y_n<0$. We claim that
\begin{equation}\label{241006-1545}
   A[\widetilde{v}]\in \p\Gamma_k\quad \text{in}~\bR^n~ \text{in the viscosity sense}.
\end{equation}

We only need to prove this in a neighborhood of every point $p$ in $\bR^n$. If $p\in\hrn$, we can obtain the desired conclusion as in Case $1$. If $p\in-\hrn$, the conclusion follows from the orthogonal invariance of $A[\widetilde{v}]\in\p\Gamma_k$. We are only left to treat the case when $p\in\p\hrn$. In the following, we assume $p=0$ since other cases follows by a translation. For $\alpha>0$ small, let 
\begin{equation*}
    \widetilde{v}_{i,\alpha}^{\pm}(y)\coloneqq \widetilde{v}_i(y)\pm  \alpha |y|^2 \pm \alpha^2 y_n.
\end{equation*}
For a function $\varphi$, denote $W[\varphi]=e^{2\varphi}A[\varphi]=-\nabla^2\varphi +\nabla \varphi\otimes\nabla\varphi -2^{-1}|\nabla \varphi|^2 I$. As in Step 1 of the proof of Lemma \ref{key-lem}, a direct computation gives 
\begin{equation*}
    W[\widetilde{v}_{i,\alpha}^\pm](y)=W[\widetilde{v}_i](y)\mp \alpha \left( 2I +O(|y|)+\alpha O(1) \right).
\end{equation*}
By \eqref{scale-d}, \eqref{equv_i}, and the cone property of $\Gamma_k$, there exist some $\alpha_0>0$ and $r_0>0$ small such that for any $0<\alpha<\alpha_0$,
\begin{equation}\label{241006-2012}
A[\widetilde{v}_{i,\alpha}^{+}]\in\bR^n\setminus\Gamma_k~\text{and}~A[\widetilde{v}_{i,\alpha}^{-}]\in\Gamma_k~ \text{in} ~\overline{B_{r_0}^+}~ \text{for}~i~ \text{large}.\end{equation}
On the other hand, a direct computation gives 
\begin{equation*}
    \p_{y_n} \widetilde{v}_{i,\alpha}^{\pm}(y)=\p_{y_n}\widetilde{v}_i(y) \pm \alpha^2.
\end{equation*}
By \eqref{boundary-normal-C1}, \eqref{scale-d}, and \eqref{equv_i}, there exists a sequence $\{\alpha_i\}$ with $\alpha_i\to 0+$ as $i\to\infty$, such that $\p_{y_n}\widetilde{v}^{+}_{i,{\alpha_i}}>0$ and $\p_{y_n}\widetilde{v}^{-}_{i,{\alpha_i}}<0$ on $\p' B_{r_0}^+$ for all $i$. Combining this with \eqref{241006-2012}, we obtain that for all $i$,
\begin{equation*}
A[\widetilde{v}_{i,{\alpha_i}}^{+}]\in\bR^n\setminus\Gamma_k~\text{and}~A[\widetilde{v}_{i,{\alpha_i}}^{-}]\in\Gamma_k~ \text{in} ~{B_{r_0}}.
\end{equation*}
Sending $i$ to infinity, by the limit property of viscosity solutions,  \eqref{241006-1545} is proved.

Using \cite[Theorem 1.4]{Li2009} (see also \cite[Theorem 1.2]{CLL-1}) and \eqref{241006-1545}, we obtain that $\widetilde{v}\equiv \text{constant}$. A contradiction. 

We reach a contradiction in both cases  and hence, we have proved $[v]_{C^\gamma(B_{1/2}^+)} \leq C$.

Now we prove the gradient estimate. Let $\widehat{v}(x) := v(x) - v(0)$. From $[\widehat{v}]_{C^\gamma(B_{1/2}^+)} = [v]_{C^\gamma(B_{1/2}^+)} \leq C$ and $\widehat{v}(0) = 0$, we obtain a two-sided bound $\sup_{B_{1/2}^+}|\widehat{v}(x)| \leq C$. Since $\widehat{v}$ satisfies
\begin{equation*}
\begin{cases}
    \sigma_k^{1/k}( A[\hat v]) = e^{2 v(0)} \quad \text{in} \,\, B^+,\\
    \cB^l[\hat{v}]=c \cdot e^{(2l-1)v(0)}\quad \text{on}~\p' B^+,
    \end{cases}
\end{equation*}
we can apply Step 1 to obtain the desired gradient bound. Theorem \ref{lgest-bk-eucildean} is proved. 
\end{proof}

\noindent
{\it On Remark \ref{remark-1.2}:}
If $\sigma_k$ in Theorem \ref{lgest-bk-eucildean} is replaced by a general $f$ as stated in Remark \ref{remark-1.2}, we modify the above proof as follows. Inequality \eqref{boundary-normal-C1} still holds since its proof only uses the boundary condition. In Step 1, by the same proof, Proposition \ref{mp-regular} remains valid if $\sigma_k$ is replaced by such $f$, and thus \eqref{boundary-tangential-C1} follows as before. 
Inequality \eqref{interior-pts-lgest} can still be obtained by Bernstein-type arguments; see \cite{Li2009}. Step 2 remains unchanged if $\sigma_k$ is replaced by such $f$.

\appendix 
\section{Proofs of Lemma \ref{linearization-bk} and \ref{corner-hopf-linear}}\label{compute-appendix}
\begin{proof}[Proof of Lemma \ref{linearization-bk}]
The proof is by the same computations as in that of \cite[Lemma 3]{wei2024}.
    A direct computation gives, on $\p'\Omega^+$,
    \begin{align}
       \cB^l[u]-\cB^l[v] &=
        \alpha(l,0) \int_0^1 \frac{d}{dt} h_{\bar{g}_{w_t}}^{2l-1}dt
        +  \sum_{s=1}^{l-1} \alpha(l,s)\int_0^1 \frac{d}{dt}\left(\sigma_s(A[w_t]^T)h_{\bar{g}_{w_t}}^{2l-2s-1}\right)dt 
         \notag \\
        &= \text{I}+ \text{II}, \label{240602-1529}
    \end{align}
    where 
    \begin{equation*}
    \text{I}\coloneqq  \sum_{s=1}^{l-1} \alpha(l,s)\int_0^1 \frac{d}{dt}\sigma_s(A[w_t]^T)h_{\bar{g}_{w_t}}^{2l-2s-1}dt,
    \end{equation*}
    and, using $h_{\bar{g}_{w_t}}=- e^{-w_t}\p_{x_n}w_t$,
\begin{gather}
\text{II} \coloneqq       \sum_{s=0}^{l-1} \alpha(l,s)\int_0^1 \sigma_s(A[w_t]^T)\frac{d}{dt}h_{\bar{g}_{w_t}}^{2l-2s-1} dt  \notag \\
 = - \int_0^1 \sum_{s=0}^{l-1} \alpha(l,s)(2l-2s-1)\sigma_s(A[w_t]^T) h_{\bar{g}_{w_t}}^{2l-2s-1} dt\cdot (u-v) \notag \\
    -\int_0^1 \sum_{s=0}^{l-1} \alpha(l,s)(2l-2s-1) \sigma_s(A[w_t]^T) h_{\bar{g}_{w_t}}^{2l-2s-2} e^{-w_t} dt \cdot \frac{\p (u-v)}{\p x_n}. \label{compute-II}
\end{gather}

By the chain rule, we have 
\begin{gather*}
 \frac{d }{d t} \sigma_s(A[w_t]^T)= -2s \sigma_s(A[w_t]^T)\cdot (u-v)\\
    + e^{-2 w_t} \frac{\p \sigma_s}{\p M_{\alpha\beta}}(A[w_t]^T) \big( 
    -\frac{\p^2 (u-v)}{\p x_\alpha\p x_\beta} +\frac{\p w_t}{\p x_\beta} \cdot \frac{\p (u-v)}{\p x_\alpha} +\frac{\p w_t}{\p x_\alpha} \cdot \frac{\p (u-v)}{\p x_\beta} \big) \\
    + e^{-2 w_t} \frac{\p\sigma_s}{\p M_{\beta\beta}}(A[w_t]^T) \big(
    -\frac{\p w_t}{\p x_\beta}\cdot \frac{\p (u-v)}{\p x_\beta} -\frac{\p w_t}{\p x_n}\cdot \frac{\p (u-v)}{\p x_n}
    \big).
\end{gather*}
Therefore, 
\begin{gather}
    \text{I}=     -\int_0^1 \sum_{s=1}^{l-1}2s \alpha(l,s) \sigma_s(A[w_t]^T) h_{\bar{g}_{w_t}}^{2l-2s-1} dt \cdot (u-v) \notag \\    
        + \int_0^1 \sum_{s=1}^{l-1}\alpha(l,s)\frac{\p \sigma_s}{\p M_{\beta\beta}}(A[w_t]^T) h_{\bar{g}_{w_t}}^{2l-2s} e^{-w_t} dt \cdot 
    \frac{\p(u-v)}{\p x_n}
    + b_\beta \frac{\p(u-v)}{\p x_\beta}
    -a_{\alpha\beta}\frac{\p^2(u-v)}{\p x_\alpha \p x_\beta}. \notag
\end{gather}

Combining the above with \eqref{240602-1529} and \eqref{compute-II}, we get the desired formula \eqref{240602-1532}: The coefficient $d_0$ is clear, and the coefficient $c_n$ is verified as below.
     \begin{gather*}
         \int_0^1 \sum_{s=1}^{l-1}\alpha(l,s)\frac{\p \sigma_s}{\p M_{\beta\beta}}(A[w_t]^T) h_{\bar{g}_{w_t}}^{2l-2s} e^{-w_t} dt \\
         =
         \int_0^1 \sum_{s=1}^{l-1}\alpha(l,s)(n-s) \sigma_{s-1}(A[w_t]^T) h_{\bar{g}_{w_t}}^{2l-2s} e^{-w_t} dt \\
         =\int_0^1 \sum_{s=0}^{l-2}\alpha(l,s+1)(n-s-1) \sigma_{s}(A[w_t]^T) h_{\bar{g}_{w_t}}^{2l-2s-2} e^{-w_t} dt.
         \end{gather*}
The first equality above uses the fact that for $M\in \cS^{(n-1)\times (n-1)}$, $\sum_{\beta=1}^{n-1} \frac{\p \sigma_s}{\p M_{\beta\beta}}(M)=(n-s)\sigma_{s-1}(M)$. Combining the above computation and the coefficient of ${\p_{x_n}(u-v)}$ in \eqref{compute-II}, we verify the coefficient $c_n$ by noting the fact $\alpha(l,s+1)(n-s-1)=\alpha(l,s)(2l-2s-1)$, $s=0,\dots,l-2$.
\end{proof}

\begin{proof}[Proof of Lemma \ref{corner-hopf-linear}]
  The proof is an adaptation of \cite[Proof of Lemma 10.1]{MR2001065}. Let $\epsi>0$ be a small constant such that $B_\epsi\cap\{x_1,x_n\geq 0\}\subset Q$. We will construct a function $\phi>0$ in $Q$ with $\frac{\p\phi}{\p x_1}(0)>0$ such that
  \begin{equation*}
  \begin{cases}
      -A_{ij}\phi_{ij} + B_i \phi_i\leq -A\phi\quad \text{in}~Q\cap B_\epsi,\\
    -a_{\alpha\beta} \phi_{\alpha\beta} + b_\beta \phi_{\beta}-c_n \phi_n \leq -A \phi \quad \text{on}~B_\epsi\cap\{x_n=0,~x_1>0\},\\ 
    \phi\leq w\quad \text{on}~\p B_\epsi\cap Q,\quad \phi=0\quad \text{on}~B_\epsi\cap \{x_1=0,~x_n\geq 0\}.
    \end{cases}
  \end{equation*}
  
  Once such $\phi$ is constructed, Lemma \ref{corner-hopf-linear} can be proved as follows. Let $\varphi=w-\phi$, then $\varphi$ satisfies 
  \begin{equation*}
  \begin{cases}
      A_{ij}\varphi_{ij} - B_i \varphi_i -A\varphi \leq 0 \quad \text{in}~Q\cap B_\epsi,\\
    a_{\alpha\beta} \varphi_{\alpha\beta} - b_\beta \varphi_{\beta}+c_n \varphi_n  -A \varphi\leq 0 \quad \text{on}~B_\epsi\cap\{x_n=0,~x_1>0\},\\ 
    \varphi\geq 0\quad \text{on}~\p B_\epsi\cap Q ~\text{and}~B_\epsi\cap \{x_1=0,~x_n\geq 0\}.
    \end{cases}
  \end{equation*}
  By the maximum principle, $\varphi\geq 0$ in $B_\epsi\cap Q$. Therefore, by $\varphi(0)=0$, we have $\frac{\p\varphi}{\p x_1}(0)\geq 0$, and thus $\frac{\p w}{\p x_1}(0)\geq \frac{\p \phi}{\p x_1}(0)>0$. 

  Finally, we construct such a function $\phi$ explicitly by setting 
  \begin{equation*}
      \phi(x)=\delta(e^{\alpha^2 x_1}-1)e^{\alpha x_n},
  \end{equation*}
  where $\alpha>0$ will be chosen as a large constant and then $\delta>0$ will be chosen small. A direct computation gives $\phi_1=\alpha^2\phi$, $\phi_n=\alpha\phi$, $\phi_{11}=\alpha^4\phi$, $\phi_{nn}=\alpha^2\phi$, and $\phi_{1n}=\alpha^3\phi$. Therefore, for large $\alpha>1$, we have 
  \begin{equation*}
      -A_{ij}\phi_{ij}+B_i\phi_i\leq \alpha^3(C_1-C_2 \alpha) \phi\leq -A\phi,
  \end{equation*}
  \begin{equation*}
      -a_{\alpha\beta}\phi_{\alpha\beta} +b_\beta\phi_\beta -c_n\phi_n \leq \alpha^2(C_3-C_4\alpha^2)\phi \leq -A\phi, 
  \end{equation*}
  where $C_m$, $m=1,\dots,4$, are some universal constants. Now $\alpha$ is a fixed amount. Since $w$ is positive continuous on $Q\setminus\{0\}$, we can choose $\delta>0$ small such that $\phi<w$ on $\p B_\epsi\cap Q$. All other properties of $\phi$ can be checked immediately. Lemma \ref{corner-hopf-linear} is proved.\end{proof}

\section{Convexity of the M\"obius Hessian}\label{convex-app}
For $n\geq 2$, let
 \begin{equation} \label{eqn-230331-0110}
   \begin{cases}
    \Gamma \subsetneqq \bR^n \,\, \text{be a non-empty open symmetric cone\footnotemark with vertex at the origin},\\
    \Gamma+\Gamma_n\subset\Gamma.
    \end{cases}
\end{equation}\footnotetext{By symmetric set, we mean that $\Gamma$ is invariant under interchange of any two $\lambda_i$.}
Denote, as in \cite{CLL-1}, $\bl\coloneqq (1,-1,\dots,-1)$.
\begin{lemma}\label{convex-lem-app}
    For $n\geq 2$, let $\Gamma$ be a convex cone satisfying \eqref{eqn-230331-0110} and $-\bl\in\overline\Gamma$, and $u,v$ be $C^2$ functions near some $x_0\in \bR^n$. If $\av,A[u]\in \Gamma$ at $x_0$, then $A[w_t]\in\Gamma$ at $x_0$, $0\leq t\leq 1$, where $w_t\coloneqq u+t(v-u)$.
\end{lemma}
The proof is similar to that of Lemma \ref{convexity-lem}. We omit the details.  

For any cone $\Gamma$ satisfying \eqref{eqn-230331-0110} with $-\bl\notin\overline\Gamma$, including $\Gamma=\Gamma_k$ for $k>n/2$ as particular cases, the conclusion of Lemma \ref{convex-lem-app} fails as shown by the following example. 

By a simple perturbation argument, we only need to find some functions $u, v$ satisfying $A[u],\av\in\overline\Gamma$ but $A[w_t]\notin \overline\Gamma$ for some $t\in[0,1]$. Let
\begin{equation*}
    u(x)=0\quad \text{and}\quad v(x)=-2\log|x|.
\end{equation*}
It is easy to check that $A[u],\av\equiv 0\in\overline\Gamma$ in $\bR^n\setminus \{0\}$. A direct computation gives 
\begin{equation*}
    e^{2w_t(x)}A[w_t](x)=2t(1-t)|x|^{-2} \left( I- 2|x|^{-2}{x}\otimes {x}  \right).
\end{equation*}
Hence, $\lambda(A[w_t])(x)=-2t(1-t)|x|^{-2}e^{-2 w_t(x)}\bl$. By $-\bl\notin\overline\Gamma$ and the cone property of $\Gamma$, we have $A[w_t]\notin\overline\Gamma$ in $\bR^n\setminus\{0\}$ for $t\in (0,1)$.

\section*{Acknowledgement}
B.Z. Chu and Y.Y. Li are partially supported by NSF Grant DMS-2247410. 
Part of this paper was completed while B.Z. Chu was visiting the City University of Hong Kong. He gratefully acknowledges the Department of Mathematics for its hospitality.

\end{document}